\documentclass[11pt]{amsart}

\usepackage{amsmath,amssymb,amsthm,mathtools}
\usepackage{mathrsfs}
\usepackage{enumitem}
\usepackage{geometry}
\usepackage{microtype}
\usepackage{hyperref}
\usepackage[pagewise]{lineno}
\usepackage{titlesec}
\titleformat{\section}{\normalfont\Large\bfseries\centering}{\thesection}{1em}{}
\newcommand{\appendixsectionformat}{%
  \titleformat{\section}
    {\normalfont\Large\bfseries\centering}
    {Appendix \thesection}
    {1em}
    {}%
}
\numberwithin{equation}{section}
\newtheorem{conjecture}{Conjecture}[section]
\newtheorem{definition}{Definition}[section]
\newtheorem{theorem}{Theorem}[section]

\newtheorem{lemma}[theorem]{Lemma}
\newtheorem{corollary}[theorem]{Corollary}
\theoremstyle{remark}

\newcommand{\R}{\mathbb{R}}

\newcommand{\dd}{\,\mathrm{d}}
\newcommand{\sgn}{\operatorname{sgn}}

\title[Limit cycles of continuous piecewise linear Li\'enard systems]{Bounds on the maximum number of limit cycles\\
of piecewise linear Li\'enard systems\\[3pt]
I. The continuous case}
\author[Cai et al.]{Shiyu Cai$^1$, Hebai Chen$^1$, Jie Jin$^1$, Yuhuan Lu$^1$}
\date{}

\address{  $^1$ School of Mathematics and Statistics, HNP-LAMA, Central South University,
	Changsha, Hunan 410083, P. R. China}

\email{cai\_shiyu@fzu.edu.cn,
chen\_hebai@csu.edu.cn,
    jin\_jie@csu.edu.cn,
	lu\_yuhuan@csu.edu.cn}



\begin{document}

	\begin{abstract}
		This paper concerns the planar Li\'enard system
		\(\dot x=y-F(x),\ \dot y=-x\), where \(F(x)\) is a continuous piecewise linear
		function with exactly \(n\) fold points. Tonnelier [SIAM J. Appl. Math. 63 (2002)] conjectured that the maximum
		number of limit cycles of the system is \(n\) when \(F\) has exactly \(n\)
		fold points. The conjecture was confirmed for \(n=1\) and \(n=2\) in
		[J. Nonlinear Sci. 25 (2015)] and
		[J. Lond. Math. Soc. 113 (2026)], respectively, whereas the case
		\(n\ge3\) remained open. In this paper, we show that the maximal number
		of limit cycles is at least \(|3n-4|\) for \(n\in\mathbb N^+\), thereby
		disproving Tonnelier's conjecture for \(n\ge3\). The proof reveals a
		unified multiscale perturbation mechanism underlying the creation and
		coexistence of multiple limit cycles. Moreover, we prove that the number of limit cycles admits the finite upper bound \(2^{28(n+1)^2}\).
	\end{abstract}
    
	\maketitle
	
	\section{Introduction and main results}
	
	The study of limit cycles is one of the central topics in the qualitative
	theory of planar differential systems. It is closely related to the second
	part of Hilbert's sixteenth problem, which asks for the maximum number and
	relative positions of limit cycles of planar polynomial differential systems
	of a given degree (see, for instance, \cite{Arnold,H,Smale}).  Despite extensive efforts over more than a century, this problem remains open
	even for quadratic polynomial systems. 

Because of the formidable difficulty of the original problem, Arnold \cite{Arnold} proposed in 1977 a weakened version of Hilbert's sixteenth
problem for small perturbations of polynomial Hamiltonian systems. Consider
\[
\dot x=H_y+\varepsilon P(x,y),\qquad
\dot y=-H_x+\varepsilon Q(x,y),
\qquad 0<|\varepsilon|\ll1,\]
where \(H,P\) and \(Q\) are polynomials. Suppose that the unperturbed
Hamiltonian system has a period annulus formed by ovals
\(
\Gamma_h\subset{H(x,y)=h},\, h\in\Sigma,
\)
where \(\Sigma\) is the interval of energy values corresponding to the period
annulus. The first-order displacement function is, up to a sign convention,
the Abelian integral
\(
I(h)=\oint_{\Gamma_h}\bigl(Q(x,y)\,dx-P(x,y)\,dy\bigr).
\)
Hence, for sufficiently small \(\varepsilon\), the bifurcating limit cycles
are governed by the isolated zeros of \(I(h)\). Arnold's weakened, or
infinitesimal, Hilbert's sixteenth problem asks for a uniform upper bound,
in terms of the polynomial degrees, on the number of such zeros.
Let \(Z(m,n)\) denote the maximal number of isolated real zeros of the
corresponding Abelian integrals, and set
\(
A(n)=Z(n+1,n).
\)
Khovanskii \cite{K} and Varchenko \cite{V} independently proved that
\(
Z(m,n)<\infty
\)
for fixed \(m\) and \(n\). Later, Binyamini, Novikov and Yakovenko \cite{BNY} obtained the explicit estimate
\(
A(n)\leq 2^{2^{\operatorname{Poly}(n)}},\)
where \(
\operatorname{Poly}(n)=O(n^p)\) and  \( p\leq61.
\)

Another influential approach, emphasized by Smale \cite{Smale,S2}, is to
restrict the original problem to special classes of polynomial differential
systems for which the limit-cycle problem remains nontrivial but is more
tractable. One of the most important such families is the polynomial
Li\'enard differential system
	\begin{equation}\label{eq}
		\dot{x} = F(x) - y, \qquad \dot{y} = x,
	\end{equation}
	where \(F(x)\) is a polynomial in \(x\).
	When \(F\) is a polynomial of degree \(n\),
	Lins, de Melo and Pugh \cite{LMP} conjectured that system
	\eqref{eq} has at most
	\(
	\left[\frac{n-1}{2}\right]
	\)
	limit cycles. 
     This conjecture has stimulated extensive research on the number and bifurcation of limit cycles of polynomial Li\'enard systems. For the low-degree cases, Lins et al.~\cite{LMP} proved that system \eqref{eq} has no limit cycles for \(n=2\) and at most one limit cycle for \(n=3\), while Li and Llibre \cite{LL} later confirmed the conjecture for \(n=4\). For \(n=5\), Li and Lu \cite{LL2} showed that the cyclicity of non-degenerate slow--fast cycles, measured through the slow divergence integral, is at most two. On the other hand,  Dumortier et al.~\cite{DM} constructed, by geometric singular perturbation methods, examples with more limit cycles than predicted by the conjecture for odd \(n\geq7\), and De Maesschalck and Huzak \cite{DM2} subsequently obtained slow--fast Li\'enard systems with \(n-2\) limit cycles for \(n\geq6\). Thus, although substantial progress has been achieved, the exact maximum number of limit cycles is still unknown in general, in particular for degrees \(n\geq5\).

	Besides smooth polynomial models, piecewise linear differential systems have
	received considerable attention because they provide a natural compromise
	between analytical tractability and rich nonlinear dynamics. They arise
	naturally in control theory, electrical circuits, mechanical systems and
	mathematical biology, and they also serve as useful approximations of smooth
	nonlinear systems. One of the most prominent examples is the piecewise
	linear approximation of the FitzHugh--Nagumo model introduced by McKean
	\cite{M1,M2}; see also Tonnelier \cite{T}.
	Although the vector field is linear in each region, switching between
	different linearity zones may produce complicated global phenomena,
	including Hopf  bifurcations, grazing bifurcations and multiple limit cycles (see \cite{CFJM,CJT,LPV,LPV2}).
	
	Under the transformation \((x,y,t)\mapsto(x,y,-t)\), system \eqref{eq} becomes
	\begin{equation}\label{sys}
		\dot x=y-F(x),\qquad \dot y=-x.
	\end{equation}
	Systems \eqref{eq} and \eqref{sys} have exactly the same number and
	configuration of limit cycles, whereas the orientation of every orbit and the stability of every limit cycle are reversed. Thus, all results concerning the number of limit cycles can be formulated equivalently for either convention. 
	
	We first introduce the terminology for the
switching values of a piecewise linear function.
	\begin{definition}\label{switching points}
		Let \(F:\mathbb{R}\to\mathbb{R}\) be a piecewise linear function, and let \(x_i\) be a switching value of \(F\), corresponding to the switching line \(x=x_i\) of system \eqref{sys}. Denote 
        \[ 
        F(x_i^-)=\lim_{x\to x_i^-}F(x), \qquad F(x_i^+)=\lim_{x\to x_i^+}F(x).
        \] 
        When \(F\) is continuous at \(x_i\), denote its one-sided
derivatives by 
        \[
        F'(x_i^-) = \lim_{h\to0^-} \frac{F(x_i+h)-F(x_i)}{h}, \qquad F'(x_i^+) = \lim_{h\to0^+} \frac{F(x_i+h)-F(x_i)}{h}. 
        \] 
		\begin{enumerate}[
    labelindent=0pt,
    labelwidth=1.5em,
    labelsep=0.4em,
    leftmargin=1.9em,
    align=left
]
			\item[\textnormal{\textbf{(i)}}] The switching value \(x_i\) is called a \emph{fold point}
			of \(F\) if
			\(F(x_i^-)=F(x_i^+)=F(x_i)\) and \(F'(x_i^-)\neq F'(x_i^+).\)
			\item[\textnormal{\textbf{(ii)}}] The switching value \(x_i\) is called a \emph{jump point}
			of \(F\) if
			\(
			F(x_i^-)\neq F(x_i^+).
			\)	
		\end{enumerate}
	\end{definition}
	Our work is concerned with the number of limit cycles of the Li\'enard system \eqref{sys},  where \(F(x)\) is continuous and piecewise linear  with exactly  \(n\) fold points.  
	This problem goes back to Tonnelier \cite{T}. In his study of a piecewise linear version of the
	FitzHugh--Nagumo system, Tonnelier was led to formulate a general conjecture relating the number of switching points of \(F\) to the number of limit cycles. In the continuous case, the conjecture can be stated as follows:
	\begin{conjecture}[Tonnelier \cite{T}]
		System \eqref{sys}, with \(F(x)\) being continuous and piecewise linear on \(n+1\) intervals, has up to \(n\) limit cycles.
	\end{conjecture}
	Llibre et al. \cite{LPV} proved that system \eqref{sys}, in which \(F(x)\) has a unique fold point and no jump points, admits at most one limit cycle. Chen et al. \cite{CFJM} further investigated the global dynamics of system \eqref{sys}, in which \(F(x)\) has two fold points and no jump points, and proved that the maximum number of limit cycles is exactly two. Therefore,  Tonnelier's conjecture has been positively resolved for \(n = 1\) and \(n = 2\). However, the cases \(n \ge 3\) had remained open. 
	A further important step was obtained recently by 
	Chen et al. \cite{CLZZ}. They proved that, for every positive integer \(n\), there exists a continuous piecewise linear function \(F\) with \(n\) fold points such that system~\eqref{sys} possesses \(n\) limit cycles. 
	The above results naturally raise the question of whether the number of
	limit cycles of system~\eqref{sys} is indeed bounded by the number of fold points of \(F\). 
    
   Let \(H_n\) denote the supremum of the number of limit cycles of system~\eqref{sys}
when \(F\) is  a continuous piecewise linear function and has exactly \(n\) fold points.
In this paper, we establish the lower bound of $|3n-4|$ for \(H_n\).
	Our results show that Tonnelier's conjecture fails for \(n\geq3\) and reveal new mechanisms responsible for the creation of multiple limit cycles.
	Furthermore, we prove that \(H_n\) admits a finite upper bound. 
Two results on the upper bound and the lower bound are respectively obtained, as stated in the following two theorems.

\begin{theorem}\label{thm:main}
		For every positive integer $n$,
		\[
		H_n\ge |3n-4|.
		\]
	\end{theorem}
From Theorem \ref{thm:main}, we obtain the following corollary.
\begin{corollary}
For \(n\geq 3\), Tonnelier's conjecture is incorrect.
\end{corollary}
    
    \begin{theorem}\label{thm:upper}
	For every positive integer $n$, 
\[
H_n\le n+
		\left\lfloor\frac{(n+2)^2}{4}\right\rfloor
		Q(n)<2^{28(n+1)^2},
\]
    where 
    \[
Q(n)=2^{8(n+1)^2-2(n+1)}10^{4(n+1)}(48n+49)^{4(n+1)}.
    \]
\end{theorem}

	The remainder of this paper is organized as follows.
Section~\ref{sec2} develops the lower-bound construction
and proves Theorem~\ref{thm:main}.
Section~\ref{sec3} establishes the upper bound and proves
Theorem~\ref{thm:upper}.
The appendices provide the supporting calculations.

	\section{The lower bound}\label{sec2}
	We first develop the lower-bound construction for \(n\ge3\).
The cases \(n=1,2\) follow from known results,
as explained in the proof of Theorem~\ref{thm:main}.

Consider system \eqref{sys}, where \(F:\R\to\R\) is
continuous and piecewise linear with exactly \(n\) fold points.
By Definition~\ref{switching points}, it can be written as
\[
F(x)=\mu x+b+\sum_{i=1}^n\omega_i|x-x_i|,
\qquad x_1<\cdots<x_n,
\]
where \(\mu,b\in\R\) and \(\omega_i\ne0\) for every \(i\).

	For the construction below, choose
	\(0<a_1<\cdots<a_n\) and weights \(e_1,\ldots,e_n\in\R\).
	For \(x\ge0\), define
	\[
	F_R(x)
	=
	x+\frac12
	\sum_{\substack{1\le i\le n\\ i\text{ odd}}}
	e_i\bigl(x-a_i+|x-a_i|\bigr),
	\qquad
	F_L(x)
	=
	x+\frac12
	\sum_{\substack{1\le i\le n\\ i\text{ even}}}
	e_i\bigl(x-a_i+|x-a_i|\bigr),
	\]
	and set
	\[
	G(x)=
	\begin{cases}
		F_R(x),& \text{if } x\ge0,\\
		-F_L(-x),& \text{if } x<0.
	\end{cases}
	\]
    
	We consider system \eqref{sys} with \(F(x)=\eta G(x)\),
	where \(\eta>0\) is sufficiently small.
	Since \(G(x)=x\) for \(|x|<a_1\), we have \(F(0)=0\).
	If \(e_i\ne0\) for every \(i\), then \(F(x)\) has exactly
	\(n\) fold points, located at \(a_i\) for odd \(i\)
	and at \(-a_i\) for even \(i\).
	
	In the polar coordinates
	\(x=r\sin\theta,\ y=r\cos\theta\),
	system \eqref{sys} with \(F=\eta G\) becomes
	\begin{equation}\label{polar}
		\dot r=-\eta\sin\theta\,G(r\sin\theta),
		\qquad
		\dot\theta
		=1-\frac{\eta}{r}\cos\theta\,G(r\sin\theta).
	\end{equation}
	Since \(G\) is globally Lipschitz and \(G(0)=0\), there exists
	\(C>0\) such that \(|G(x)|\le C|x|\). Hence, for sufficiently small
	\(\eta>0\) and \(r>0\),
	\[
	\dot\theta
	=1-\frac{\eta}{r}\cos\theta\,G(r\sin\theta)
	\ge 1-\eta C>0.
	\]
	Thus every orbit starting from the positive \(y\)-axis rotates
	monotonically around the origin and intersects the positive \(y\)-axis
	again when \(\theta=2\pi\).
	Let \(r(\theta;r_0,\eta)\) denote its radial coordinate, with
	\(r(0;r_0,\eta)=r_0>0\).
		Define the return map and its displacement by
	\[
	P_\eta(r_0)=r(2\pi;r_0,\eta),
	\qquad
	\Delta_\eta(r_0)=P_\eta(r_0)-r_0.
	\]
	The following lemma gives the first-order expansion
of \(\Delta_\eta\).
	\begin{lemma}\label{lm:average}
		For system \eqref{sys} with \(F(x)=\eta G(x)\) defined above,
		the displacement function \(\Delta_\eta\) satisfies
		\begin{equation}\label{avg}
			\Delta_\eta(r)=-\eta r\Phi(r)+O(\eta^2),
		\end{equation}
		as \(\eta\to0^+\), uniformly for \(r\) in compact subsets of
	\((0,\infty)\), where
		\begin{equation}\label{Phi}
			\Phi(r)=\pi+\sum_{i=1}^n e_iK(a_i/r),
			\qquad
			K(z)=
			\begin{cases}
				\arccos z-z\sqrt{1-z^2},& \text{if } 0\le z<1,\\
				0,& \text{if } z\ge1.
			\end{cases}
		\end{equation}
	\end{lemma}
\begin{proof}
	We first derive the first-order averaged displacement and then compute the
	corresponding averaged integral. Since \(\dot\theta>0\), we may use
	\(\theta\) as the independent variable. From the polar form 
	\eqref{polar},
	\[
	\frac{\dd r}{\dd\theta}
	=
	-\frac{\eta\sin\theta\,G(r\sin\theta)}
	{1-\frac{\eta}{r}\cos\theta\,G(r\sin\theta)}
	=
	-\eta\sin\theta\,G(r\sin\theta)+O(\eta^2),
	\]
	uniformly for \(r\) in compact subsets of \((0,\infty)\).
	The function \(-\sin\theta\,G(r\sin\theta)\) is
	\(2\pi\)-periodic in \(\theta\) and globally Lipschitz in \(r\).
	Thus, the first-order averaging formula
	\cite[Theorem~1.1]{BuicaLlibre} gives
	\begin{equation}\label{first-average}
		P_\eta(r)-r
		=-\eta\int_0^{2\pi}\sin\theta\,G(r\sin\theta)\,\dd\theta
		+O(\eta^2).
	\end{equation}
	
	It remains to compute
	\(-\int_0^{2\pi}\sin\theta\,G(r\sin\theta)\dd\theta\).
	The linear term of \(G\) contributes
	\(-\int_0^{2\pi}r\sin^2\theta\,\dd\theta=-\pi r\).
	For odd \(i\), the term involving \(e_i\) vanishes when
	\(r\le a_i\). If \(r>a_i\), set \(z=a_i/r\) and
	\(\alpha=\arcsin z\). Its contribution is
	\[
	-e_i\int_\alpha^{\pi-\alpha}
	\sin\theta\,(r\sin\theta-a_i)\dd\theta
	=-e_i r\left(\arccos z-z\sqrt{1-z^2}\right)
	=-e_i rK(z).
	\]
	For even \(i\), the substitution \(\theta\mapsto\theta+\pi\)
	gives the same contribution. Therefore,
	\[
	-\int_0^{2\pi}\sin\theta\,G(r\sin\theta)\,\dd\theta
	=-r\left(\pi+\sum_{i=1}^n e_iK(a_i/r)\right)
	=-r\Phi(r).
	\]
	Substituting this identity into \eqref{first-average}
	gives \eqref{avg} and proves the lemma.
\end{proof}

	For each integer \(j\geq0\), let
	\begin{equation}\label{Mj}
		M_j=\sum_{i=1}^n e_i a_i^j.
	\end{equation}
	The following conditions on \(M_j\) control the sign changes of \(\Phi\).
	
	\begin{lemma}\label{lm:finite}
		If
		\begin{equation}\label{moment-condition}
			M_0=-2,\qquad M_1=M_3=\cdots=M_{2n-3}=0,
		\end{equation}
		then \(\Phi\) has at least \(n-2\) positive sign-changing zeros.
	\end{lemma}
\begin{proof}
	We first show that \eqref{moment-condition} implies
	\begin{equation}\label{orthogonality}
		\int_0^\infty r^{2j+3}\Phi(r)\dd r=0,
		\qquad j=0,\ldots,n-3,
	\end{equation}
	and then use \eqref{orthogonality} to control the number of positive
	sign-changing zeros of \(\Phi\).
	
	Indeed, differentiating \eqref{Phi} gives
	\begin{equation}\label{Phi-prime}
		\Phi'(r)
		=\frac{2}{r^3}
		\sum_{i=1}^n e_i a_i
		\sqrt{\frac{r^2-a_i^2+|r^2-a_i^2|}{2}}.
	\end{equation}
	Hence, by \eqref{Phi-prime} and integration by parts, for \(R>a_n\),
	\begin{equation}\label{IBP-Phi}
		\int_0^R r^{2j+3}\Phi(r)\dd r
		=
		\frac{R^{2j+4}\Phi(R)}{2j+4}
		-\frac{2}{2j+4}\sum_{i=1}^n e_i a_iJ_j(a_i,R),
	\end{equation}
	where
	\[
	J_j(a,R)
	=
	\int_a^R r^{2j+1}\sqrt{r^2-a^2}\dd r
	=
	\sum_{k=0}^j
	\binom{j}{k}a^{2(j-k)}
	\frac{(R^2-a^2)^{k+3/2}}{2k+3}.
	\]
	Every term involving a nonnegative power of \(R\) in the expansion of \(J_j(a,R)\) is of the form
	\(C_{j\ell}a^{2\ell}R^{2j+3-2\ell}\), with \(0\le\ell\le j+1\).
	Therefore, the terms involving nonnegative powers of \(R\) in
	\(\sum_{i=1}^n e_i a_iJ_j(a_i,R)\) are linear combinations of
	\(M_{2\ell+1}R^{2j+3-2\ell}\), \(0\le\ell\le j+1\), and hence vanish
	because \(M_1=M_3=\cdots=M_{2j+3}=0\). Consequently,
	\[
	\sum_{i=1}^n e_i a_iJ_j(a_i,R)=O(R^{-1})
	\qquad \text{as} \quad R\to\infty.
	\]
	
	On the other hand, under \eqref{moment-condition}, the expansion of \(K\)
	at the origin gives \(\Phi(R)=O(R^{-(2n-1)})\). Since \(j\le n-3\),
	\[
	R^{2j+4}\Phi(R)
	=O\!\left(R^{2j-2n+5}\right)
	=O(R^{-1}).
	\]
	Thus \eqref{IBP-Phi} gives
	\(\int_0^R r^{2j+3}\Phi(r)\dd r=O(R^{-1})\), and letting
	\(R\to\infty\) proves \eqref{orthogonality}.
	
	Second, we use \eqref{orthogonality} to obtain a lower bound for the number
	of sign changes of \(\Phi\). Suppose that \(\Phi\) has only \(k \le n-3\) positive sign-changing zeros, denoted by \(r_1 < \cdots < r_k\). Define
	\[
	\Psi(r)=r^3\Phi(r)\prod_{\ell=1}^k(r^2-r_\ell^2).
	\]
	Since each \(r_\ell\) is a sign-changing zero of \(\Phi\), both \(\Phi(r)\)
	and \(r^2-r_\ell^2\) change sign when \(r\) crosses \(r_\ell\). Hence
	these sign changes cancel, and \(\Psi\) has a constant sign on
	\((0,\infty)\). Moreover, since \(K(a_i/r)=0\) for \(0<r<a_1\), we have
	\(\Phi(r)=\pi\) on this interval, and hence \(\Psi\not\equiv0\). Since
	\(k\le n-3\) and \(\Phi(r)=O(r^{-(2n-1)})\), the function \(\Psi\) is
	integrable on \((0,\infty)\). Therefore,
	\[
	\int_0^\infty \Psi(r)\dd r\ne0.
	\]
	In addition, expanding the product shows that this integral is a
	linear combination of \(\int_0^\infty r^{2j+3}\Phi(r)\dd r\),
	\(j=0,\ldots,k\), all of which vanish by \eqref{orthogonality}. This is a
	contradiction. Thus \(\Phi\) has at least \(n-2\) positive sign-changing
	zeros. This proves Lemma~\ref{lm:finite}.
\end{proof}

	We shall use the following parameter values in our construction:
	\begin{equation}\label{trig-base}
		a_i=\sin\left(\frac{(2i-1)\pi}{4n}\right),
		\qquad
		e_i=4(-1)^i\cos\left(\frac{\pi}{4n}\right)
		\cos\left(\frac{(2i-1)\pi}{4n}\right),
		\qquad 1\le i\le n.
	\end{equation}
	We further introduce the following quantities associated with the return map at infinity. Define
	\[
	Q_k=\int_0^{a_n}x^{2k+1}
	\left(F_R(x)^2-F_L(x)^2\right)\dd x,
	\qquad
	N_{m\ell}
	=
	\int_0^{a_n}x^{2(m-\ell)-1}
	\left(I_{\ell,R}(x)^2-I_{\ell,L}(x)^2\right)\dd x,
	\]
	where
	\[
	I_{\ell,R}(x)=\int_0^x t^{2\ell+1}F_R(t)\dd t,
	\qquad
	I_{\ell,L}(x)=\int_0^x t^{2\ell+1}F_L(t)\dd t.
	\]
	Set
	\begin{equation}\label{Bk}
		B_k=-Q_k+
		\sum_{\substack{m+\ell=k-1\\m>\ell}}
		2(m-\ell)^2c_mc_\ell N_{m\ell},
	\end{equation}
	where \(c_j=\binom{2j}{j}/4^j\) for \(j\ge0\)
	and an empty sum is understood to be zero.
	
	The following lemma establishes the degeneracy conditions satisfied by
	\(M_j\) and \(B_k\) at the parameter values given in \eqref{trig-base}.
	
	\begin{lemma}
		\label{lm:degeneracy}
		At the parameter values given in \eqref{trig-base},
		\[
		M_0=-2,\qquad M_1=M_3=\cdots=M_{2n-3}=0.
		\]
		Moreover, \(B_{n-3}\) is the first nonzero member of the sequence
		\(\{B_k\}_{k\ge0}\), and
\begin{equation}\label{sgn-critical-B}
\sgn B_{n-3}=(-1)^n.
\end{equation}
	\end{lemma}
	\begin{proof}
	We first verify the stated identities for \(M_j\) at the parameter values
	\eqref{trig-base}, then prove that \(B_0=\cdots=B_{n-4}=0\), and finally
	determine the sign of \(B_{n-3}\).
	
	For convenience, set
\[
\vartheta=\frac{\pi}{2n},
\qquad
h=\frac{\pi}{4n},
\qquad
s=\sin\frac{\pi}{4n},
\qquad
q=\cos\frac{\pi}{2n},
\qquad
t_i=\frac{(2i-1)\pi}{4n}.
\]
	By \eqref{trig-base}, we have
	\(a_i=\sin t_i\), \(e_i=4(-1)^i\cos h\cos t_i\).
	Hence, by \eqref{Mj},
	\[
	M_0=\sum_{i=1}^n e_i
	=4\cos h\sum_{i=1}^n(-1)^i\cos t_i=-2,
	\]
	and for \(j\ge1\),
	\[
	M_{2j-1}
	=4\cos h\sum_{i=1}^n(-1)^i\sin^{2j-1}t_i\cos t_i.
	\]
	
	By the double-angle formulas, the function
	\(\sin^{2j-1}t\cos t\) lies in
	\(\operatorname{span}\{\sin(2t),\sin(4t),\ldots,\sin(2jt)\}\).
	Hence the identity
	\[
	\sum_{i=1}^n(-1)^i\sin(2\ell t_i)
	=
	\begin{cases}
		0, & \text{if } 1\le \ell<n,\\[1mm]
		-n, &  \text{if } \ell=n,
	\end{cases}
	\]
	implies
	\[
	M_1=M_3=\cdots=M_{2n-3}=0.
	\]
	At the critical order \(j=n\), the
	coefficient of \(\sin(2nt)\) in \(\sin^{2n-1}t\cos t\) is
	\((-1)^{n+1}2^{1-2n}\), and hence
	\begin{equation}\label{critical-M}
		M_{2n-1}=(-1)^n n\,2^{3-2n}\cos h.
	\end{equation}
	
	Define
	\[
	S(x)=F_R(x)+F_L(x),\qquad
	D(x)=F_R(x)-F_L(x).
	\]
	We next derive their interval expressions used in the calculation of \(B_k\).
A direct calculation shows that for \(a_j<x<a_{j+1}\),
\begin{equation}\label{SDj}
	S(x)=2(-1)^j\cos(j\vartheta)
	\left(x-\frac{\cos h\,\sin(j\vartheta)}{q}\right),
	\qquad
	D(x)=\frac{\sin^2(j\vartheta)-2\cos h\,\sin(j\vartheta)x}{s}.
\end{equation}
	Extend \(S\) and \(D\) from
	\([0,a_n]\) to \([-a_n,a_n]\) by
	\[
	S(-x)=-S(x),\qquad D(-x)=D(x),
	\qquad 0\le x\le a_n.
	\] 
	We partition \([-a_n,a_n]\) into the intervals
	\[
	I_j=
	\begin{cases}
		[-a_{1-j},-a_{-j}], & \text{if } -n+1\le j\le -1,\\[1mm]
		[-a_1,a_1], & \text{if }  j=0,\\[1mm]
		[a_j,a_{j+1}], & \text{if } 1\le j\le n-1.
	\end{cases}
	\]
	Since \(a_j=\sin(j\vartheta-h)\), these intervals can be written uniformly as
	\[
	I_j=
	[\sin(j\vartheta-h),\,\sin(j\vartheta+h)],
	\qquad j=-n+1,\ldots,n-1.
	\]
	For convenience, we also set
	\(I_{-n}=\{-a_n\}\), a degenerate interval of zero length. For the interval \(I_j\), set
\[
Q_{k,j}=\int_{I_j}x^{2k+1}S(x)D(x)\,\dd x,
\qquad
N_{m\ell,j}
=\int_{I_j}x^{2(m-\ell)-1}
I_{\ell,S}(x)I_{\ell,D}(x)\,\dd x,
\]
where
\[
I_{\ell,S}(x)=\int_0^x t^{2\ell+1}S(t)\,\dd t,
\qquad
I_{\ell,D}(x)=\int_0^x t^{2\ell+1}D(t)\,\dd t.
\]

By \eqref{SDj}, after extracting the factor \((-1)^j\),
the coefficients of \(x^2,x,1\) in \(S(x)D(x)\) on \(I_j\)
have Fourier degrees at most \(2,3,4\), respectively.
After multiplication by \(x^{2k+1}\) and integration over \(I_j\),
the resulting Fourier degrees are at most
\(2+(2k+4)\), \(3+(2k+3)\), and \(4+(2k+2)\).
Hence
\[
\deg_{\rm F}Q_{k,j}\le 2k+6,
\]
where
\(\deg_{\rm F}\) denotes the Fourier degree in \(j\vartheta\).

The calculation in Appendix~\ref{app:fourier-calculations}
(see \eqref{app:N-fourier-degree}) gives, for \(m+\ell=k-1\),
\[
\deg_{\rm F}N_{m\ell,j}\le2k+6,
\]
after extracting the factor \((-1)^j\).

After extracting \((-1)^j\), the symmetry
\((j\vartheta,x)\mapsto(j\vartheta+\pi,-x)\)
shows that the interval contributions are \(\pi\)-periodic in
\(j\vartheta\). Thus \(Q_{k,j}\) and \(N_{m\ell,j}\) are
\((-1)^j\) times \(\pi\)-periodic trigonometric polynomials of
Fourier degree at most \(2k+6\), and hence are linear combinations of
\((-1)^j e^{2irj\vartheta}\) with \(|r|\le k+3\).
Since \(\vartheta=\pi/(2n)\), the geometric-series formula gives
\begin{equation}\label{fourier-cancellation}
	\sum_{j=-n}^{n-1}(-1)^j e^{2irj\vartheta}
	=
	(-1)^{-n}e^{-2irn\vartheta}
	\frac{1-\bigl(-e^{2ir\vartheta}\bigr)^{2n}}
	{1+e^{2ir\vartheta}}
	=0,
	\qquad |r|<n.
\end{equation}
Since \(S\) and \(I_{\ell,S}\) are odd, while
\(D\) and \(I_{\ell,D}\) are even, we have
\[
\sum_{j=-n}^{n-1}Q_{k,j}=2Q_k,
\qquad
\sum_{j=-n}^{n-1}N_{m\ell,j}=2N_{m\ell}.
\]
Therefore, all contributions vanish for \(k\le n-4\), and hence
\[
B_0=\cdots=B_{n-4}=0.
\]
It remains to determine the sign of \(B_{n-3}\). 
    
    The calculation in Appendix~\ref{app:fourier-calculations}
(see \eqref{A-critical-Q}) gives
\begin{equation}\label{critical-Q}
-Q_{n-3}
=
\frac{n-1-2q^2}
{sq(n-1)(2n-4)(2n-3)}M_{2n-1}.
\end{equation}
	The coefficient of \(M_{2n-1}\)  is strictly positive for every \(n\ge3\). Hence, by \eqref{critical-M} and \eqref{critical-Q}, we have 
	\begin{equation}\label{sgn-critical-Q}
		\sgn(-Q_{n-3})=\sgn(M_{2n-1})=(-1)^n.
	\end{equation}
	
    For \(n=3,4\), the sum in \eqref{Bk} is empty for \(k=n-3\), so
	\eqref{sgn-critical-Q} already gives \(\sgn B_{n-3}=(-1)^n\).
	Assume therefore
	\(n\ge5\). Since all coefficients of \(N_{m\ell}\) in \eqref{Bk} are
	positive, it suffices to show that
	\begin{equation}\label{sgn-critical-N}
		\sgn N_{m\ell}=(-1)^n
	\end{equation}
	for every \(m>\ell\) with \(m+\ell=n-4\). 
    
    Set
	\(d=2\ell+4\) and \(\rho=\cos h/q\). Then
	\(4\le d\le n-1\). The highest Fourier coefficients associated with
\(I_{\ell,S}\) and \(I_{\ell,D}\), computed in
Appendix~\ref{app:fourier-calculations},
are expressed in terms of the polynomials
	\[
	\begin{aligned}
		f_S(w)
		&=\frac{w^{d-1}}{d-1}
		-\rho\frac{w^{d-2}}{d-2}
		+\frac{\cos h}{(d-1)(d-2)\cos(dh)},\\
		f_D(w)
		&=\frac{2\cos h\,w^{d-1}}{d-1}
		-\frac{w^{d-2}}{d-2}
		+\frac{2s\cos h}{(d-1)(d-2)\sin(dh)}.
	\end{aligned}
	\]
	Let \(H_S(t)=e^{-idt}f_S(e^{it})=U(t)+iV(t)\) and
	\(H_D(t)=e^{-idt}f_D(e^{it})=X(t)+iY(t)\). A direct calculation gives
	\begin{align}
		U''+d^2U
		&=(d+1)\cos t-\rho(d+2)\cos(2t),
		&U'(0)&=0,\quad U(h)=0,
		\label{U-BVP}\\
		Y''+d^2Y
		&=2\sin t\bigl((d+2)\cos t-(d+1)\cos h\bigr),
		&Y(0)&=Y(h)=0.
		\label{Y-BVP}
	\end{align}
	
	We first show that
	\[
	U>0,\qquad V<0,\qquad X>0,\qquad Y<0
	\quad\text{on }(0,h).
	\]
	Indeed, \(\cos t/\cos(2t)\) is increasing on \([0,h]\) and
	\(\rho\cos(2t)\ge\cos t\), so the right-hand side of
	\eqref{U-BVP} is negative, whereas that of \eqref{Y-BVP} is
	positive. Since \(dh<\pi/2\), the corresponding Green kernels of
	\(-\partial_t^2-d^2\) are positive, and hence \(U>0\) and \(Y<0\).
	Furthermore,
	\[
	V'=\cos t-\rho\cos(2t)-dU<0,\qquad
	X'=dY+2\sin t(\cos h-\cos t)<0.
	\]
	Since \(V(0)=0\), we obtain \(V<0\); and since
	\[
	(d-1)(d-2)X(h)
	=d-1-2\cos^2h+2s\cos h\cot(dh)>0,
	\]
	we obtain \(X>0\).
	
	For \(0<t<h\), we have \(0<2nt<\pi/2\), and hence
	\[
	\operatorname{Re}(e^{2nit}H_D(t))=X\cos(2nt)-Y\sin(2nt)>0.
	\]
	Moreover, we claim that
	\begin{equation}\label{Im-HD-positive}
		\begin{aligned}
			\operatorname{Im}(e^{2nit}H_D(t))
			=\frac{2\cos h\,\sin((2n-1)t)}{d-1}
			-\frac{\sin((2n-2)t)}{d-2}
			+\frac{2s\cos h\,\sin((2n-d)t)}
			{(d-1)(d-2)\sin(dh)}
			>0.
		\end{aligned}
	\end{equation}
	Indeed, since
	\(0<(2n-2)t<(2n-1)t<\pi/2\),
	we have
	\(\sin((2n-1)t)>\sin((2n-2)t)>0\).
	Moreover,
	\(2\cos h(d-2)>d-1\), because
	\(\cos h=\cos(\pi/(4n))\ge\cos(\pi/20)>3/4\) for \(n\ge5\),
	and the difference between the two sides is increasing in \(d\ge4\).
	Thus the first two terms in \eqref{Im-HD-positive} have positive sum,
	while the last term is positive, proving \eqref{Im-HD-positive}. Consequently,
	\[
	\mathcal I
	:=\int_0^h
	\operatorname{Re}\bigl(e^{2nit}H_S(t)H_D(t)\bigr)\dd t
	=
	\int_0^h
	\left[
	U\operatorname{Re}(e^{2nit}H_D(t))
	-
	V\operatorname{Im}(e^{2nit}H_D(t))
	\right]\dd t
	>0.
	\]
	
The calculation in
Appendix~\ref{app:fourier-calculations}
(see \eqref{app:critical-N}) gives
\[
N_{m\ell}
=
(-1)^n\frac{2n}{s\,2^{2n-2}}\mathcal I.
\]
	Since \(s>0\), this proves \eqref{sgn-critical-N}. Together with \eqref{sgn-critical-Q}, this yields \eqref{sgn-critical-B}. This proves Lemma~\ref{lm:degeneracy}.
	\end{proof}

	Define
\[
\mathcal D(\mathbf e,\mathbf a)=
\begin{cases}
    \bigl(M_0+2,M_1,M_3\bigr),
    & \text{if } n=3,\\[1mm]
    \bigl(M_0+2,M_1,M_3,\ldots,M_{2n-3},
    B_0,\ldots,B_{n-4}\bigr),
    & \text{if } n\ge4.
\end{cases}
\]
The next lemma shows that the components of \(\mathcal D\)
can be varied independently through changes in
\((\mathbf e,\mathbf a)\).

\begin{lemma}\label{lm:rank}
    At the parameter values given in \eqref{trig-base},
    the Jacobian matrix of \(\mathcal D\) with respect to
    \((\mathbf e,\mathbf a)\) has rank \(2n-3\).
\end{lemma}
\begin{proof}
	We first show that the Jacobian of
\(M_0+2,M_1,M_3,\ldots,M_{2n-3}\) with respect to
\(\mathbf e\) is nonsingular. For \(n\ge4\), we then
show that the differentials of \(B_0,\ldots,B_{n-4}\),
restricted to the manifold defined by
\eqref{moment-condition}, are linearly independent.
	
	The Jacobian of \(M_0+2,M_1,M_3,\ldots,M_{2n-3}\) with respect to
	\(\mathbf e=(e_1,\ldots,e_n)\) is
	\[
	\frac{\partial(M_0+2,M_1,M_3,\ldots,M_{2n-3})}
	{\partial(e_1,\ldots,e_n)}
	=
	\begin{pmatrix}
		1&1&\cdots&1\\
		a_1&a_2&\cdots&a_n\\
		a_1^3&a_2^3&\cdots&a_n^3\\
		\vdots&\vdots&&\vdots\\
		a_1^{2n-3}&a_2^{2n-3}&\cdots&a_n^{2n-3}
	\end{pmatrix}.
	\]
	Since \(0<a_1<\cdots<a_n\), this generalized Vandermonde matrix is
	nonsingular. Hence these \(n\) components of \(\mathcal D\) have full
	rank with respect to \(\mathbf e\).
	
	For \(n=3\), this proves Lemma~\ref{lm:rank}. We therefore assume that
	\(n\ge4\). Since the above Jacobian is nonsingular, the equations
	\eqref{moment-condition} define a smooth \(n\)-dimensional manifold \(\mathcal M\) in the \((\mathbf e,\mathbf a)\)-parameter space
near the parameter values given in \eqref{trig-base}..
	Moreover, \(a_1,\ldots,a_n\) can be used as local coordinates on
	\(\mathcal M\). Let \(D_i\) denote differentiation with respect to
	\(a_i\) along \(\mathcal M\). All parameter derivatives below are evaluated at
the parameter values \eqref{trig-base}.
	
	We claim that, for \(0\le k\le n-4\),
	\begin{equation}\label{gradient-poly}
		D_iB_k=e_iP_k(a_i^2),
		\qquad
		\deg P_k=k+2.
	\end{equation}
	Once \eqref{gradient-poly} is established, the differentials of
	\(B_0,\ldots,B_{n-4}\) are linearly independent on \(\mathcal M\).
	Indeed, any nontrivial linear relation among them would give a nonzero
	polynomial of degree at most \(n-2\) vanishing at the \(n\) distinct
	points \(a_1^2,\ldots,a_n^2\), which is impossible.
	
	The calculations in Appendix~\ref{app:constrained-derivatives} (see \eqref{app:Q-gradient} and \eqref{app:N-gradient}) give, for \(0\le k\le n-4\),
\[
D_iQ_k=e_iP_Q(a_i^2),
\qquad
\deg P_Q\le k+2.
\]
For \(m+\ell=k-1\) with \(m>\ell\), they also give
\[
D_iN_{m\ell}=e_iP_N(a_i^2),
\qquad
\deg P_N\le k+2.
\]
Here the dependence of \(P_Q\) and \(P_N\) on their indices
is suppressed. Let \(\kappa_Q\) and \(\kappa_N\) denote
the coefficients of \(t^{k+2}\) in \(P_Q(t)\) and \(P_N(t)\),
respectively.

	We next compute \(\kappa_Q\) and \(\kappa_N\).
	Recall that \(\vartheta=\pi/(2n)\), and define
	\[
	T(x)
	=
	\frac{x\sin((x-2)\vartheta)-(x-2)\sin(x\vartheta)}
	{(x-1)(x-2)\sin(x\vartheta)},
	\qquad
	F_*(x)
	=
	\frac{T(x)-1}
	{\sin(\vartheta/2)\cos\vartheta}.
	\]
	The calculations in Appendix~\ref{app:constrained-derivatives} (see \eqref{app:kappa-Q} and \eqref{app:kappa-N}) give
\[
\kappa_Q=F_*(2k+5)-F_*(2k+4)
\]
and, for \(m+\ell=k-1\) with \(m>\ell\),
\[
\kappa_N
=-\Bigl(
F_*[2\ell+4,2m+4,2k+5]
-
F_*[2\ell+4,2m+4,2k+4]
\Bigr),
\]
where \(F_*[\,\cdot,\cdot,\cdot\,]\) denotes the second
divided difference of \(F_*\).
	
	We next determine their signs. Put \(z=x-1\). Using the Euler product
	for \(\sin x\), we obtain
	\[
	T(1+z)
	=
	\sum_{i\ge1}
	A_i(z)\prod_{j<i}\bigl(1+zA_j(z)\bigr),
	\]
	where
	\[
	A_i(z)
	=
	\frac{4\vartheta^2}
	{\pi^2i^2-(z+1)^2\vartheta^2},
	\qquad i\ge1.
	\]
	Every Taylor coefficient of \(A_i\) at \(z=0\) is strictly positive,
	and the series above converges locally uniformly for
\(|z|<2n-1\). Hence for $1<x<2n$,
	\(T'(x)>0\), \(T'''(x)>0\). Since \(\sin(\vartheta/2)\cos\vartheta>0\), the same inequalities
hold for \(F_*\).

All nodes occurring in the formulas for
\(\kappa_Q\) and \(\kappa_N\) lie in
\([4,2n-3]\subset(1,2n)\). Thus the inequality \(F_*'(x)>0\) gives \(\kappa_Q>0\).
	Moreover, \(F_*'''(x)>0\) implies that the second divided difference
	of \(F_*\) is strictly increasing in its third node, and hence
	\(\kappa_N<0\).
	
	By \eqref{Bk}, the \(Q_k\)-coefficient is negative and the
	\(N_{m\ell}\)-coefficients are positive, so the leading coefficient
	of the reduced gradient of \(B_k\) is negative. Thus \eqref{gradient-poly} holds. Consequently,
	the differentials of \(B_0,\ldots,B_{n-4}\) are linearly independent
	on \(\mathcal M\), and the Jacobian matrix of \(\mathcal D\) with respect to
\((\mathbf e,\mathbf a)\) has rank \(2n-3\).
	This proves Lemma~\ref{lm:rank}.
	\end{proof}

	Let \(\widehat F_R\) and \(\widehat F_L\) be the affine
	functions agreeing with \(F_R\) and \(F_L\), respectively,
	for \(x\ge a_n\).
	
	To study the return map near infinity, define
\[
\widehat G(x)=
\begin{cases}
\widehat F_R(x),& \text{if } x\ge0,\\
-\widehat F_L(-x),& \text{if }x<0.
\end{cases}
\]
Let \(P_R,P_L\) denote the half-return maps for \(F=\eta G\) in the right and left half-planes, respectively, and let
\(\widehat P_R,\widehat P_L\) denote the corresponding maps
for \(F=\eta\widehat G\). 

Then we relate 
	\(P_R,P_L\) to \(\widehat P_R,\widehat P_L\) by conjugacy.
We construct the conjugacies using the transition equations
\[
\frac{\dd y}{\dd x}=-\frac{x}{y-\eta G(x)},
\qquad
\frac{\dd y}{\dd x}=-\frac{x}{y-\eta\widehat G(x)}.
\]
For sufficiently large \(|y_0|\), the solutions with
\(y(0)=y_0\), taken separately in the two half-planes,
are defined up to \(x=\pm a_n\) and satisfy
\(y(x)=y_0+O(|y_0|^{-1})\) uniformly in \(x\).
For the first equation, denote their values at \(x=a_n\)
and \(x=-a_n\) by \(T_R(y_0)\) and \(T_L(y_0)\), respectively,
and define \(\widehat T_R,\widehat T_L\) analogously for
the second equation.
    
	These transition maps are locally invertible by uniqueness
	of solutions and their dependence on \(\eta\) is suppressed. Define the corresponding conjugating maps
	\[
	C_R=\widehat T_R^{-1}\circ T_R,
	\qquad
	C_L=\widehat T_L^{-1}\circ T_L.
	\]
    They conjugate the half-return maps according to
\begin{equation}\label{half-return-conjugacy}
P_R=C_R^{-1}\circ\widehat P_R\circ C_R,
\qquad
P_L=C_L^{-1}\circ\widehat P_L\circ C_L .
\end{equation}
	In the coordinate \(u=1/y\), the transition maps and
	their inverses are analytic near \(u=0\).
	Returning to \(y\), we obtain the expansions
	\[
	C_R(y)=y+\sum_{j\ge1}A_{j,R}y^{-j},
	\qquad
	C_L(y)=y+\sum_{j\ge1}A_{j,L}y^{-j}.
	\]
	
	The following lemma relates the coefficients in these
	expansions to \(M_j\) and \(B_k\).
	
	\begin{lemma}\label{lm:return-coefficients}
		If \(M_0+2=M_1=0\), then
		\begin{equation}\label{Aeven}
A_{2j,R}-A_{2j,L}
=-\frac{c_{j-1}}{(2j)(2j+1)}
\eta M_{2j+1}+O(\eta^3),
\qquad 1\le j\le n-1.
\end{equation}
		Moreover, under \eqref{moment-condition}, we have
		\begin{equation}\label{Aodd}
			A_{2k+3,R}-A_{2k+3,L}
			=\eta^2B_k+O(\eta^4),
			\qquad 0\le k\le n-3.
		\end{equation}
	\end{lemma}
\begin{proof}
	Integrating \(-x/y+O(|y|^{-2})\) from \(0\) to the corresponding
endpoint \(x=\pm a_n\) gives \(y-a_n^2/(2y)+O(|y|^{-2})\)
for all four transition maps \(T_R,T_L,\widehat T_R,\widehat T_L\). Consequently, the common \(y^{-1}\) terms in
	\(C_R\) and \(C_L\) vanish, giving \(A_{1,R}=A_{1,L}=0\).
	
	If \(M_0+2=M_1=0\), we have \(\widehat F_R(x)+\widehat F_L(x)=(2+M_0)x-M_1=0\).
	Expanding the conjugating maps to first order in \(\eta\) therefore gives
	\[
	\left.
	\frac{\partial}{\partial\eta}
	\bigl(A_{2j,R}-A_{2j,L}\bigr)
	\right|_{\eta=0}
	=-c_{j-1}\int_0^{a_n}x^{2j-1}
	\bigl(F_R(x)+F_L(x)\bigr)\dd x
	=-\frac{c_{j-1}}{(2j)(2j+1)}M_{2j+1},
	\]
	which yields the coefficient of \(\eta\) in \eqref{Aeven}. The \(O(\eta^3)\) remainder follows from the parity in \(\eta\),
	which we establish below.
	
	We next compute the coefficient of \(\eta^2\) in \eqref{Aodd}
	under \eqref{moment-condition}.
	Introduce \(r=\sqrt{x^2+y^2}\), so that \(r=y\)
	at \(x=0\) for \(y>0\). Then for \(f\in\{F_R,F_L,\widehat F_R,\widehat F_L\}\),
	\[
	\frac{\dd r}{\dd x}
	=-\eta\frac{xf(x)}{r\sqrt{r^2-x^2}}
	-\eta^2\frac{xf(x)^2}{r(r^2-x^2)}+O(\eta^3).
	\]
	Write its transition map as
	\(W_f(\eta,y)=y+\eta W_{1,f}(y)+\eta^{2} W_{2,f}(y)+O(\eta^3)\), and
	put \(I_{\ell,f}(x)=\int_0^x t^{2\ell+1}f(t)\dd t\).
	Expanding the integral equation for \(r\) to second order
	in \(\eta\) and then applying the binomial expansion
	in powers of \(y^{-1}\), we obtain
	\[
	\begin{aligned}
		W_{1,f}(y)&=-\sum_{j\ge0}c_jI_{j,f}(a_n)y^{-2j-2},\\
		[y^{-2k-3}]W_{2,f}
		&=-\int_0^{a_n}x^{2k+1}f(x)^2\dd x-\sum_{m+\ell=k-1}(2m+2)c_mc_\ell
		\int_0^{a_n}I'_{m,f}(x)I_{\ell,f}(x)\dd x,
	\end{aligned}
	\]
	where \([y^{-2k-3}]W_{2,f}(y)\) denotes the coefficient of
	\(y^{-2k-3}\) in the expansion of \(W_{2,f}(y)\) at infinity.
	By \eqref{moment-condition}, we have
	\[
	I_{\ell,R}(a_n)+I_{\ell,L}(a_n)
	=\frac{M_{2\ell+3}}{(2\ell+2)(2\ell+3)}=0,
	\qquad 0\le\ell\le n-3.
	\]
	Consequently, \(I_{m,R}(a_n)I_{\ell,R}(a_n)
	=I_{m,L}(a_n)I_{\ell,L}(a_n)\).
	For \(m>\ell\), using
	\(I'_{m,R}=x^{2(m-\ell)}I'_{\ell,R}\) and
	\(I'_{m,L}=x^{2(m-\ell)}I'_{\ell,L}\), we obtain
	\[
	\begin{aligned}
		\int_0^{a_n}
		\bigl(I'_{m,R}I_{\ell,R}-I'_{m,L}I_{\ell,L}\bigr)\dd x&=\frac12\int_0^{a_n}x^{2(m-\ell)}
		\bigl(I_{\ell,R}^2-I_{\ell,L}^2\bigr)'\dd x\\
		\quad&=\frac12
		\Bigl[x^{2(m-\ell)}
		\bigl(I_{\ell,R}^2-I_{\ell,L}^2\bigr)\Bigr]_0^{a_n}
		-(m-\ell)N_{m\ell}\\
		\quad	&=-(m-\ell)N_{m\ell}.
	\end{aligned}
	\]
	Integration by parts also gives
	\[
	\int_0^{a_n}
	\bigl(I'_{m,R}I_{\ell,R}-I'_{m,L}I_{\ell,L}\bigr)\dd x
	=-\int_0^{a_n}
	\bigl(I'_{\ell,R}I_{m,R}-I'_{\ell,L}I_{m,L}\bigr)\dd x.
	\]
	Thus, grouping the terms with indices \((m,\ell)\) and
	\((\ell,m)\) in the sum for \(W_{2,f}\), we obtain
	\[
	[y^{-2k-3}](W_{2,F_R}-W_{2,F_L})
	=-Q_k+\sum_{\substack{m+\ell=k-1\\m>\ell}}
	2(m-\ell)^2c_mc_\ell N_{m\ell}=B_k.
	\]
	
We now transfer this coefficient identity to \(C_R-C_L\).
Reflecting the left half-plane by \(\xi=-x\) and using
the fact that the radial coordinate change at \(x=\pm a_n\)
is independent of \(f\), we have
\[
C_R=W_{\widehat F_R}(\eta,\cdot)^{-1}
\circ W_{F_R}(\eta,\cdot),
\qquad
C_L=W_{\widehat F_L}(-\eta,\cdot)^{-1}
\circ W_{F_L}(-\eta,\cdot).
\]
Using \(\widehat F_R+\widehat F_L=0\),
\(W_{1,-f}=-W_{1,f}\), and \(W_{2,-f}=W_{2,f}\),
we obtain the following expansion.
	\[
	\begin{aligned}
		C_R(y)-C_L(y)
		&=\eta\bigl(W_{1,F_R}+W_{1,F_L}\bigr)\\
		&\quad+\eta^2\Bigl[
		W_{2,F_R}(y)-W_{2,F_L}(y)
		-W'_{1,\widehat F_R}(y)
		\bigl(W_{1,F_R}(y)+W_{1,F_L}(y)\bigr)
		\Bigr]+O(\eta^3).
	\end{aligned}
	\]
	By \eqref{moment-condition},
	\(W_{1,F_R}+W_{1,F_L}=O(y^{-2n+2})\).
	Together with \(W'_{1,\widehat F_R}=O(y^{-3})\), this gives
	\[
	W'_{1,\widehat F_R}(y)
	\bigl(W_{1,F_R}(y)+W_{1,F_L}(y)\bigr)
	=O(y^{-2n-1}).
	\]
	Thus the correction contributes no \(y^{-2k-3}\) term
	for \(0\le k\le n-3\), and the coefficient of \(\eta^2\)
	in \(A_{2k+3,R}-A_{2k+3,L}\) is \(B_k\).
	
	Splitting the transition at the fold points shows that
	the coefficients depend analytically on
	\((\eta,\mathbf e,\mathbf a)\).
	Since \(C_R=C_L=\mathrm{id}\) at \(\eta=0\), all these
	coefficients vanish there.
	Finally, the symmetry \((y,\eta)\mapsto(-y,-\eta)\)
	makes \(A_{2j,R}\), \(A_{2j,L}\) odd for \(j\ge1\)
	and \(A_{2j+1,R}\), \(A_{2j+1,L}\) even for \(j\ge0\)
	in \(\eta\).
	This gives the remainders in \eqref{Aeven} and \eqref{Aodd}
	and completes the proof.
	\end{proof}
	
	Lemmas \ref{lm:average}-\ref{lm:return-coefficients} will be used to construct coexisting
	limit cycles at finite amplitudes and near infinity,
	yielding Theorem \ref{thm:main}.

\noindent{\bf Proof of Theorem~\ref{thm:main}.}
	First, consider the case $n=1$. In \cite{LPV}, Llibre, Ponce and Valls 
	proved that system \eqref{sys} with
	\[
	F(x)=
	\begin{cases}
		a_lx, & \text{if } x\leq x_1,\\
		a_r(x-x_1)+a_lx_1, & \text{if } x>x_1,
	\end{cases}
	\]
	can admit one limit cycle, where $(a_l,a_r,x_1)\in\mathbb{R}^3$.
	Hence, Theorem \ref{thm:main} holds for $n=1$.
	
	Next, consider the case $n=2$. According to \cite{CFJM,CJT,CLXY,LPV} and the
	references therein, there exists a system \eqref{sys}, with $F(x)$ piecewise
	linear on three intervals and having two fold points, that admits
	exactly two limit cycles. Hence, Theorem \ref{thm:main} also holds for $n=2$.
	
	Assume therefore that \(n\ge3\). Consider system
	\eqref{sys} with \(F(x)=\eta G(x)\).
	
	We first obtain periodic
	orbits at finite amplitudes, then construct additional limit
	cycles near infinity, and finally make the selected periodic
	orbits hyperbolic by a small perturbation.
	
	By Lemmas~\ref{lm:finite} and~\ref{lm:degeneracy},
\(\Phi\) has at least \(n-2\) positive sign-changing zeros
at the parameter values in \eqref{trig-base}.
Choose pairwise disjoint compact intervals
\([r_j^-,r_j^+]\subset(0,\infty)\), \(j=1,\ldots,n-2\),
on each of which \(\Phi\) has opposite endpoint signs.

By continuity in \((\mathbf e,\mathbf a)\), these signs persist
in a sufficiently small fixed neighborhood of the parameter
values in \eqref{trig-base}.
On this neighborhood, the functions \(G\) have uniform
Lipschitz and linear-growth bounds, so the proof of
Lemma~\ref{lm:average} gives a uniform \(O(\eta^2)\) remainder
at the finitely many selected endpoints.

Hence, for all sufficiently small \(\eta>0\),
\(\Delta_\eta\) has opposite signs at the endpoints of each
interval and therefore has a zero in each interval,
corresponding to a periodic orbit.
All subsequent parameter perturbations are chosen to
preserve these endpoint signs.
	
	To construct additional limit cycles coexisting with the
	periodic orbits in these intervals, we next construct a degenerate parameter curve using
	Lemma~\ref{lm:return-coefficients}.
	
	By Lemma~\ref{lm:rank}, select \(2n-3\) coordinates of
	\((\mathbf e,\mathbf a)\), collected into a parameter vector
	\(\mathbf p\), such that the Jacobian of \(\mathcal D\) with
	respect to \(\mathbf p\) is nonsingular at the parameter
	values in \eqref{trig-base}. Let \(\mathbf p_0\) denote the
	corresponding value of \(\mathbf p\). Fix the remaining
	three coordinates at their values in \eqref{trig-base}
	and let \(\mathbf p\) vary near \(\mathbf p_0\).
	
	We seek a parameter curve through \(\mathbf p_0\) along
	which \(M_0+2=M_1=0\) and
	\begin{equation}\label{conjugacy-conditions}
A_{j,R}-A_{j,L}=0,
\qquad j=2,\ldots,2n-4.
\end{equation}
	 Divide the equations with even \(j\) by \(\eta\),
and those with odd \(j\) by \(\eta^2\).
	At \(\eta=0\), the equations \(M_0+2=M_1=0\), together
	with the even-indexed equations after division, give
	\eqref{moment-condition} by \eqref{Aeven}.
	On the moment manifold \(\mathcal M\), the odd-indexed equations,
	after division, reduce to \(B_0=\cdots=B_{n-4}=0\)
	by \eqref{Aodd}.
	Elementary row operations therefore transform the Jacobian
	of the full normalized system with respect to \(\mathbf p\),
	evaluated at \(\eta=0\) and \(\mathbf p=\mathbf p_0\),
	into \(\partial_{\mathbf p}\mathcal D(\mathbf p_0)\),
	up to nonzero row factors.
	The latter is nonsingular by the choice of \(\mathbf p\).
	
	As shown in the proof of Lemma~\ref{lm:return-coefficients},
	after division by \(\eta\) and \(\eta^2\), respectively,
	the left-hand sides of \eqref{conjugacy-conditions}
	extend analytically to \(\eta=0\) as functions of
	\(\eta^2\) and \(\mathbf p\).
	Hence the implicit function theorem yields a parameter
	curve, analytic in \(\eta^2\),
	\begin{equation}\label{pstar}
		\mathbf p^*(\eta)=\mathbf p_0+O(\eta^2),
	\end{equation}
	along which \(M_0+2=M_1=0\) and
	\eqref{conjugacy-conditions} hold.
	Since \(\mathbf p^*(\eta)\to\mathbf p_0\) as \(\eta\to0^+\),
	the endpoint signs on the fixed compact intervals persist
	along this curve for sufficiently small \(\eta>0\).
	
	Denote the corresponding \(G\) by \(G_{\mathbf p}\), and
	let \(P_{\eta,\mathbf p}\) be the return map associated
	with \(F_{\eta,\mathbf p}=\eta G_{\mathbf p}\).
	With \(u=1/r\), the displacement function near infinity
\[
U(u;\eta,\mathbf p)
=u\left(P_{\eta,\mathbf p}(1/u)-\frac1u\right)
=\alpha_0+\alpha_1u+\sum_{j\ge2}\alpha_j u^j
\]
	is analytic near \(u=0\), where
	\(\alpha_j=\alpha_j(\eta,\mathbf p)\).
    
For sufficiently small \(\eta>0\) and parameters near
\(\mathbf p_0\), the equilibria of the affine systems
corresponding to \(\widehat F_R\) and \(\widehat F_L\)
are foci lying on the section \(x=0\).
Thus the corresponding half-return maps
\(\widehat P_R\) and \(\widehat P_L\) are affine.

By the conjugacy relations
\eqref{half-return-conjugacy} and the composition \(P_{\eta,\mathbf p}=P_{\eta,\mathbf p,L}\circ P_{\eta,\mathbf p,R }\), Lemma~\ref{lm:return-coefficients} implies that the expansion
of \(P_{\eta,\mathbf p}(r)\) at infinity contains no
\(r^{-1}\) term. Therefore,
\(\alpha_2\equiv0\).
	
	Set
	\[
	\sigma(\mathbf p)
	=
	1+\sum_{\substack{1\le i\le n\\ i\text{ odd}}}e_i,
	\qquad
	\kappa=\exp\left(
	-\frac{\pi\eta \sigma(\mathbf p)}{\sqrt{4-\eta^2 \sigma(\mathbf p)^2}}
	\right).
	\]
	At \(\mathbf p=\mathbf p_0\), we have \(\sigma(\mathbf p_0)=-\cot(\pi/(4n))<0\),
	so \(\kappa>1\) near \(\mathbf p_0\) for sufficiently small \(\eta>0\).
	
A direct computation of \(\widehat P_R\) and \(\widehat P_L\)
shows that \(M_0+2=M_1=0\) implies \(\alpha_0=\alpha_1=0\).
Moreover, \(\alpha_2\equiv0\), as shown above.
For each \(j\ge2\), after the preceding coefficients
\(\alpha_0,\ldots,\alpha_j\) have been eliminated,
\eqref{half-return-conjugacy} and the corresponding expansions
of \(C_R,C_L\) give
\begin{equation}\label{alpha-A}
\alpha_{j+1}
=\bigl(1+(-1)^j\kappa^{-j-1}\bigr)
\bigl(A_{j,R}-A_{j,L}\bigr).
\end{equation}
Thus \eqref{conjugacy-conditions} and \eqref{alpha-A} give
	\begin{equation}\label{alpha-J-zero}
		\alpha_j(\eta,\mathbf p^*(\eta))=0,
		\qquad j=0,\ldots,2n-3.
	\end{equation}
The successive relations between the normalized equations
and these coefficients are triangular, with \(n\) diagonal
factors of order \(\eta\) and \(n-3\) of order \(\eta^3\).

Since \(\kappa>1\), the factors
\(1+(-1)^j\kappa^{-j-1}\) in \eqref{alpha-A} do not vanish,
so all diagonal factors are nonzero for sufficiently
small \(\eta>0\).
Hence
	\begin{equation}\label{coefficient-rank}
		\operatorname{rank}
		\left.
		\frac{\partial(\alpha_0,\alpha_1,\alpha_3,\ldots,\alpha_{2n-3})}
		{\partial\mathbf p}
		\right|_{\mathbf p=\mathbf p^*(\eta)}
		=2n-3.
	\end{equation}
	
	We next determine \(\alpha_{2n-2}\) and
\(\alpha_{2n-1}\) along \eqref{pstar}. By \eqref{Aodd}, \eqref{alpha-A}, and \eqref{pstar},
	we obtain
	\begin{equation}\label{alpha-m}
		\alpha_{2n-2}(\eta,\mathbf p^*(\eta))=b_n\eta^3+O(\eta^4),
	\end{equation}
	where 
	\(b_n=-(n-1)\pi \sigma(\mathbf p_0)B_{n-3}(\mathbf p_0)\).
	The first-order averaged displacement gives
	\begin{equation}\label{alpha-mplus1}
		\alpha_{2n-1}(\eta,\mathbf p^*(\eta))=d_n\eta+O(\eta^2),
	\end{equation}
where \(d_n=-2c_{n-1}M_{2n-1}(\mathbf p_0)/((2n-3)(2n-1))\).

	Since \(\sigma(\mathbf p_0)<0\),
\eqref{sgn-critical-B} and \eqref{critical-M} give \(\sgn b_n=(-1)^n\) and
	\(\sgn d_n=(-1)^{n+1}\). In particular, \(b_nd_n<0\).
	
	We now construct the additional periodic orbits near infinity.
	By analyticity, \(U(u;0,\mathbf p)\equiv0\), and
	\eqref{alpha-J-zero}, along \eqref{pstar} we have
	\[
	U(u;\eta,\mathbf p^*(\eta))
	=u^{2n-2}\bigl(\alpha_{2n-2}+\alpha_{2n-1}u+\eta u^2R(u,\eta)\bigr),
	\]
	where \(R\) is analytic and the coefficients are evaluated
	at \(\mathbf p=\mathbf p^*(\eta)\). Substituting \(u=\eta^2v\) and using
	\eqref{alpha-m}--\eqref{alpha-mplus1} yields
	\[
	\eta^{-4n+1}U(\eta^2v;\eta,\mathbf p^*(\eta))
	\longrightarrow v^{2n-2}(b_n+d_nv)
	\qquad\text{as }\eta\to0^+,
	\]
	with convergence in \(C^1\) on every compact interval
	contained in \((0,\infty)\). The positive simple zero \(v_*=-b_n/d_n\)
	therefore gives a hyperbolic limit cycle whose
\(u\)-coordinate \(u_*\) satisfies
	\(u_*\sim\eta^2v_*\), and hence \(r_*=1/u_*\sim\eta^{-2}/v_*\).

	Thus for sufficiently small \(\eta>0\), this limit cycle is distinct from the \(n-2\) periodic orbits
	previously obtained in fixed bounded intervals of \(r\).
	
	Fix such an \(\eta\).  Define
\[
	Q(w)=
	\left(
	w+\frac{\displaystyle\sum_{j=1}^{2n-3}j^{-1}}
	{\displaystyle\sum_{1\le i<j\le 2n-3}(ij)^{-1}}
	\right)
	\prod_{j=1}^{2n-3}(w-j).
\]
	This monic polynomial has zero coefficient of \(w^2\) and
	the positive simple zeros \(1,\ldots,2n-3\).
    
    For another small parameter
\(\delta>0\), introduce the second scale \(u=\delta w\). Then we have
\[
	\delta^{-2n+2}U(\delta w;\eta,\mathbf p)
	=
	\sum_{j\ge0}
	\alpha_j(\eta,\mathbf p)\delta^{j-2n+2}w^j.
\]
Thus, to obtain the limiting polynomial
\(\alpha_{2n-2}(\eta,\mathbf p^*(\eta))Q(w)\), we prescribe
\[
	\alpha_j(\eta,\mathbf p)
	=
	\alpha_{2n-2}(\eta,\mathbf p^*(\eta))
	\delta^{2n-2-j}[w^j]Q(w),
	\qquad j=0,1,3,\ldots,2n-3,
\]
where \([w^j]Q(w)\) denotes the coefficient of \(w^j\) in \(Q(w)\).
The choice \([w^2]Q(w)=0\) is compatible with
\(\alpha_2\equiv0\), so no condition is needed for \(j=2\). Since the right-hand sides tend to zero as
\(\delta\to0^+\), \eqref{coefficient-rank} and the inverse function
theorem give parameters
\(\mathbf p_{\eta,\delta}\to\mathbf p^*(\eta)\)
satisfying these equations. Consequently,
	\[
	\delta^{-2n+2}U(\delta w;\eta,\mathbf p_{\eta,\delta})
	\longrightarrow
	\alpha_{2n-2}(\eta,\mathbf p^*(\eta))Q(w)
	\qquad\text{as }\delta\to0^+,
	\]
	with convergence in \(C^1\) on every compact interval
	contained in \((0,\infty)\).
	Since \(\alpha_{2n-2}(\eta,\mathbf p^*(\eta))\ne0\),
	there are \(2n-3\) further hyperbolic periodic
	orbits with \(u_j\sim\delta j\), and hence
	\(r_j\sim1/(\delta j)\). 
	
	Since \(\mathbf p_{\eta,\delta}\to\mathbf p^*(\eta)\), the previously
	constructed hyperbolic limit cycle persists, and the
	endpoint signs ensuring the \(n-2\) finite-amplitude
	periodic orbits remain unchanged. Taking \(\delta\ll\eta^2\), we have \(u_j\ll u_*\), so the new
limit cycles lie outside the hyperbolic limit cycle with \(u_*\sim\eta^2v_*\).
	
	Finally, \eqref{trig-base} gives \(e_i\ne0\) for every 	\(i\), so all slope jumps remain nonzero under the preceding 	small perturbations. Now perturb \(F\) by \(F_\varepsilon(x)=F(x)+\varepsilon x\), and let \(P_\varepsilon\) denote the corresponding return map. 
    
    Each trajectory has finitely many switching contacts per turn, since
tangencies satisfy \(\ddot x=-x\ne0\) in the original time variable and
are therefore isolated.

Write the radial equation as
\[
\frac{dr}{d\theta}=\mathcal R(\theta,r,\varepsilon)
:=-\frac{\sin\theta\,F_\varepsilon(r\sin\theta)}
{1-\cos\theta\,F_\varepsilon(r\sin\theta)/r}.
\]
Since $\mathcal R$ is continuous and each trajectory has only
finitely many switching contacts, the coefficients in the
difference-quotient equations converge away from these contacts.

On the relevant compact set, for sufficiently small
$|\varepsilon|$, the denominator in $\mathcal R$ has a uniform
positive lower bound, so its piecewise first derivatives are
uniformly bounded. Dominated convergence and Gr\"onwall's
inequality then yield the variational equations and joint
$C^1$ dependence of $P_\varepsilon(r_0)$ on
$(r_0,\varepsilon)$, including at tangencies.
Moreover, 
\[
\partial_\varepsilon\mathcal R(\theta,r,\varepsilon)
=-\frac{r\sin^2\theta}
{\bigl(1-\cos\theta\,F_\varepsilon(r\sin\theta)/r\bigr)^2},
\]
and variation of constants yields
\[
\partial_\varepsilon P_\varepsilon(r_0)
=\int_0^{2\pi}
\exp\!\left(\int_s^{2\pi}\partial_r\mathcal R\,d\theta\right)
\partial_\varepsilon\mathcal R\,ds<0.
\]
Here all derivatives are evaluated along
$r(\theta;r_0,\varepsilon)$; their values at the finitely many
switching contacts do not affect the integrals.

Thus, the implicit function theorem locally writes the zero set
of $P_\varepsilon(r)-r$ as $\varepsilon=h(r)$, and
\[
h'(r)
=-\frac{P_{h(r)}'(r)-1}
{\left.\partial_\varepsilon P_\varepsilon(r)
\right|_{\varepsilon=h(r)}}.
\]
By the one-dimensional Sard theorem, applied to a countable cover
by these local graphs, an arbitrarily small $\varepsilon$ can be
chosen as a common regular value. Hence all fixed points in the
selected intervals are simple.

The endpoint signs ensure at least one such fixed point in each
interval, while the other hyperbolic periodic orbits persist.
Since adding \(\varepsilon x\) does not affect the fold points,
\(F_\varepsilon\) still has exactly \(n\) fold points.
The corresponding system therefore admits at least
\((n-2)+1+(2n-3)=3n-4\) hyperbolic limit cycles.
Hence
\[
H_n\ge3n-4,
\]
proving the theorem.
	\hfill\(\Box\)

A numerical phase portrait is presented to illustrate
Theorem~\ref{thm:main} in the case \(n=3\).
In this case, the theorem guarantees the existence of parameter values
for which system~\eqref{sys} admits at least five limit cycles.
For the numerical illustration, consider the system
\begin{equation}\label{sys:n3}
    \dot x=y-F(x),\qquad \dot y=-x,
\end{equation}
where
\[
F(x)=
\begin{cases}
    m_1x+d_1, & x\leq 1,\\
    m_2x+d_2, & 1<x\leq 2,\\
    m_3x+d_3, & 2<x\leq 3,\\
    m_4x+d_4, & x>3,
\end{cases}
\]
with
\[
\begin{aligned}
    m_1&=0.1,
    &d_1&=0,\\
    m_2&=-0.40000434999952982325,
    &d_2&=0.50000434999952982325,\\
    m_3&=0.00000434999982673535,
    &d_3&=-0.30001304999918329395,\\
    m_4&=-0.1000000000000000293,
    &d_4&=2.97\times10^{-13}.
\end{aligned}
\]
These coefficients ensure that \(F\) is continuous at
the three switching points \(x=1,2,3\).
To display the five cycles together despite their widely differing
amplitudes, the nonlinear coordinate transformation
\(X=\operatorname{arcsinh}(x)\), \(Y=\operatorname{arcsinh}(y)\) is applied.
The resulting phase portrait is shown in Figure~\ref{tu},
where the square-like appearance of the cycles is a consequence
of this coordinate compression.
\begin{figure}[htpb]
\centering
{\includegraphics[width=0.8\textwidth]
{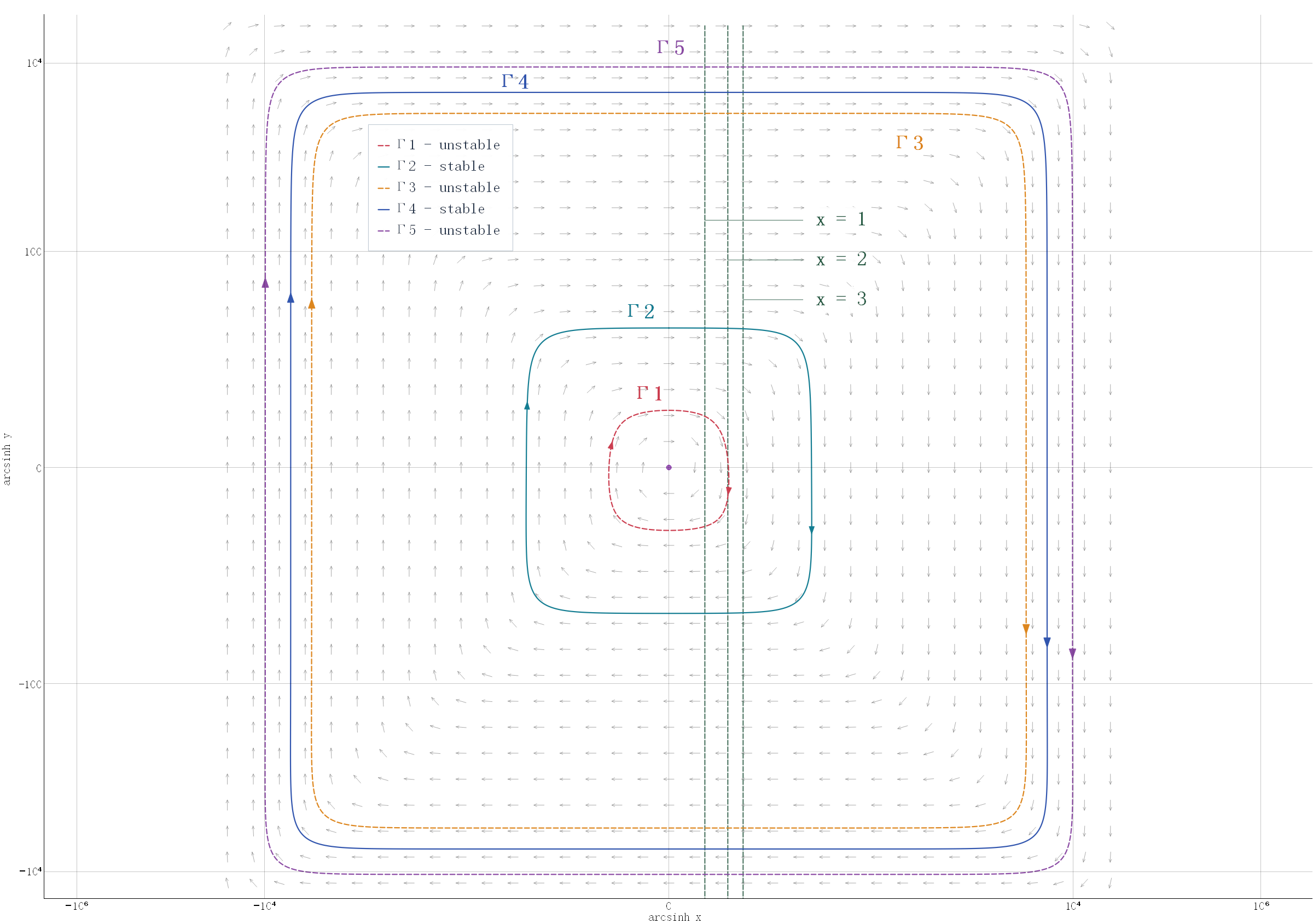}}   
\caption{Numerical phase portrait of system \eqref{sys:n3}. }
\label{tu}
\end{figure}

	\section{The upper bound}\label{sec3}
	
In this section, \(F\) is an arbitrary continuous piecewise
linear function with exactly \(n\) fold points, where \(n\ge1\).

We first bound the possible crossing patterns of periodic
orbits, then describe each pattern by a finite system of
Pfaffian equations, and finally bound its isolated real zeros.

	\begin{lemma}\label{lm:upper-geometry}
		Every  periodic orbit  of system \eqref{sys} intersects the \(y\)-axis
		exactly twice. On the orbit segments in its right and left half-planes,
		\(x(t)\) has exactly one critical point, a maximum and a minimum, respectively.
	\end{lemma}
\begin{proof}
For a periodic solution of period \(T\), $\int_0^T x(t)\,dt=-\bigl[y(T)-y(0)\bigr]=0$, so its orbit enters both half-planes \({x<0}\) and \({x>0}\); otherwise $x(t)\equiv0$ and the solution would be constant.

Since \((0,F(0))\) is an  equilibrium, a  periodic orbit cannot pass through this point. 
On the \(y\)-axis, \(\dot x=y-F(0)\), so all intersections
above \((0,F(0))\) are rightward crossings, whereas all
those below it are leftward crossings.

Since the orbit is a simple closed curve, the crossing
directions alternate when the intersection points are
ordered along the \(y\)-axis, by the Jordan curve theorem.
Hence there is at most one intersection above \((0,F(0))\)
and at most one below it.
As the orbit
meets both half-planes, both crossings must occur. Therefore, it
intersects the \(y\)-axis exactly twice.

		Put $v=\dot x$. If $v(t_0)=0$, then $x(t_0+h)-x(t_0)=o(h)$.
		By the Lipschitz continuity of $F$,
		\[
		v'(t_0)=\dot y(t_0)
		-\lim_{h\to0}\frac{F(x(t_0+h))-F(x(t_0))}{h}
		=-x(t_0).
		\]
Every zero of $v$ on the right orbit segment is therefore a positive to negative crossing, so there is at most one. Two such zeros would require a crossing in the opposite direction between them. There is at least one because $x>0$ in the interior and $x=0$ at both endpoints. The left orbit segment is treated similarly. Each half-plane orbit segment therefore has exactly one critical point.   
\end{proof}

A periodic orbit is called \emph{tangent} if it is tangent
to a switching line, and \emph{transverse} otherwise.
For a transverse periodic orbit, cut it at every intersection
with a switching line or the $y$-axis, counting the line
$x=0$ only once. We call this the \emph{canonical decomposition}.
Starting at the upper intersection with the $y$-axis and
following the flow, the \emph{itinerary} records the successive
cutting lines and the affine region containing each orbit segment.

The following lemma bounds the number of tangent periodic
orbits and the number of possible itineraries of transverse
periodic orbits.
\begin{lemma}\label{lm:upper-itineraries}
		System~\eqref{sys} has at most \(n\) tangent periodic orbits.
The canonical decomposition of every transverse periodic
orbit consists of at most \(2n+2\) orbit segments, and the number of
possible itineraries is at most
\(\lfloor(n+2)^2/4\rfloor\).
	\end{lemma}
	\begin{proof}
For system	\eqref{sys}, a periodic orbit tangent to \(x = x_i\) satisfies \(\dot{x} = 0\) there, so the tangency occurs at \((x_i, F(x_i))\). By uniqueness of solutions, at most one periodic
orbit can pass through this point. If \(x_i=0\), this point is \((0,F(0))\), which is an equilibrium and hence cannot belong to a
periodic orbit. Therefore, the number of tangent periodic orbits is
at most \(n\).
        
		Let \(p\) and \(q\) denote the numbers of positive and negative
fold points, respectively. By Lemma~\ref{lm:upper-geometry}, each
nonzero switching line met by a transverse periodic orbit is crossed
exactly twice. If \(j_+\) and \(j_-\) denote the numbers of crossed
positive and negative lines, respectively, then its canonical
decomposition has
\[
2+2j_++2j_-\le2+2(p+q)\le2n+2
\]
orbit segments. Turning points are not cuts, and each open orbit segment lies in one
affine region and one open half-plane.

On each side the crossed lines form an initial segment ordered from the axis outward, so $(j_+,j_-)$ determines the entire itinerary, giving at most $(p+1)(q+1)$ possibilities. Since the fold point at the origin, if
present, does not generate an additional cutting line, we have
\(p+q\le n\). Therefore,
\[
(p+1)(q+1)
\le \frac{(p+q+2)^2}{4}
\le \frac{(n+2)^2}{4}.
\]
Taking the integer part gives the stated bound.
\end{proof}

\begin{definition}
Let \(U\subset\mathbb R^d\) be open, with
\(\mathbf z=(z_1,\ldots,z_d)\). An ordered family of real analytic
functions \((\varphi_1,\ldots,\varphi_\ell)\) on \(U\) is called a
Pfaffian chain of length \(\ell\) and degree at most \(\alpha\) if
\[
	\partial_{z_i}\varphi_j
	=
	P_{ij}(\mathbf z,\varphi_1,\ldots,\varphi_j),
\]
where each \(P_{ij}\) is a polynomial of total degree at most \(\alpha\). A function $f=P(\mathbf z,\varphi_1,\ldots,\varphi_\ell)$ is Pfaffian of degree at most $\beta$ with respect to the chain
\((\varphi_1,\ldots,\varphi_\ell)\) if $P$ is a polynomial of total degree at most $\beta$. 
\end{definition}
 All chains used below are defined on the whole ambient space.
Khovanskii's theorem, in the form stated in
\cite[Theorem~3.1]{Gab}, gives an explicit bound on the number
of common zeros at which the Jacobian is nonsingular.
The following lemma extends this to all
isolated common zeros.
	\begin{lemma}\label{lm:upper-isolated}
Let \((f_1,\ldots,f_s)\) be Pfaffian functions of degree at most
\(\beta\ge1\) on \(\R^d\) with respect to the same Pfaffian chain
of length \(\ell\) and degree at most \(\alpha\ge1\). Their number of isolated common real zeros is at most
		\begin{equation}\label{upper-isolated-bound}
			2^{\ell(\ell-1)/2}D^d
			\bigl(\alpha\min\{d,\ell\}+dD-d+1\bigr)^\ell,
		\end{equation}
        where \(D=2\beta+\alpha-1\).
	\end{lemma}
	
	\begin{proof}
Choose $N$ isolated common zeros $z^{(1)},\ldots,z^{(N)}$ and pairwise disjoint closed balls $B_j$ centered at $z^{(j)}$, with radii $r_j$, containing no other common zeros.
Set $H=\sum_{i=1}^s f_i^2$ and $\delta_j=\min_{\partial B_j}H>0$. For $\boldsymbol{\varepsilon} \in\mathbb R^d$ satisfying $\|\boldsymbol{\varepsilon}\|<\min_j\delta_j/(2r_j)$, we have
\[
\bigl(H(\mathbf z)-\boldsymbol{\varepsilon}\cdot \mathbf z\bigr)
-\bigl(H(z^{(j)})-\boldsymbol{\varepsilon}\cdot z^{(j)}\bigr)
\ge \delta_j-\|\boldsymbol{\varepsilon}\|r_j>0,
\qquad z\in\partial B_j.
\]
Thus, by compactness, $H-\boldsymbol{\varepsilon}\cdot \mathbf z$ attains its minimum at an interior point of each $B_j$. By Sard's theorem, choose such an $\boldsymbol{\varepsilon}$ to be a regular value of $\nabla H$.
Then
\[
\#\{\mathbf z:\nabla H(\mathbf z)=\boldsymbol{\varepsilon}\}\geq N,
\]
and every solution of \(\nabla H(\mathbf z)=\boldsymbol{\varepsilon}\) satisfies
\(
\det\nabla^2H(\mathbf z)\ne0.
\)
Since $\partial_{z_j}H=2\sum_i f_i\partial_{z_j}f_i$, each component of $\nabla H-\boldsymbol{\varepsilon}$ has degree at most $\beta+(\beta+\alpha-1)=D$ in the same chain \cite[Lemma~2.5]{Gab}.
Applying Khovanskii's bound \cite[Theorem~3.1]{Gab} to these \(d\) Pfaffian functions gives
\[
N\leq
2^{\ell(\ell-1)/2}D^d
\bigl(\alpha\min\{d,\ell\}+dD-d+1\bigr)^\ell.
\]
 Since $N$ was arbitrary, 
the same bound holds for the total number.
	\end{proof}

To apply the zero-counting bound, we first express the endpoint matching conditions for individual orbit segments as Pfaffian equations. Since $F$ is affine on each region determined by the switching lines, we consider
\begin{equation}\label{eq:upper-affine}
\dot x=y-ax-b,\qquad \dot y=-x,
\end{equation}
where $a,b\in\mathbb R$ are fixed. We consider orbit segments of system \eqref{eq:upper-affine} whose interiors lie in one of the half-planes $x>0$ or $x<0$. 

Let $(\xi_0,Y_0)$ and $(\xi_1,Y_1)$ be the endpoints of such a segment, with passage time $\tau>0$, and put $z_i=(\xi_i,Y_i-b)^{\mathsf T}$ for $i=0,1$. The endpoint matching conditions are
\begin{equation}\label{eq:upper-arc-flow}
z_1=e^{\tau A}z_0,\qquad A=\begin{pmatrix}-a&1\\-1&0\end{pmatrix}.
\end{equation}
The characteristic polynomial of $A$ is $\lambda^2+a\lambda+1$, with discriminant $a^2-4$.

With $\xi_0,\xi_1,a,b$ fixed, we seek a global Pfaffian representation of the matching conditions in $Y_0,Y_1$ and a time parameter obtained from $\tau$, with a common chain depending only on that parameter. All degrees below are total degrees in these variables and the chain functions.

\begin{lemma}\label{lm:upper-arc}
The matching conditions \eqref{eq:upper-arc-flow} admit such a Pfaffian representation, with chain length $\ell$, chain degree at most $\alpha$, and equation degree at most $\beta$ satisfying:
\begin{enumerate}[
    labelindent=0pt,
    labelwidth=1.5em,
    labelsep=0.4em,
    leftmargin=1.9em,
    align=left
]
\item[\textnormal{\textbf{(i)}}] If $a^2>4$, then $\ell\le2$, $\alpha\le1$, and $\beta\le2$.
\item[\textnormal{\textbf{(ii)}}] If $a^2=4$, then $\ell\le1$, $\alpha\le1$, and $\beta\le3$.
\item[\textnormal{\textbf{(iii)}}] If $a^2<4$, then $\ell\le2$, $\alpha\le3$, and $\beta\le4$.
\end{enumerate}
The chosen time parameter is a local analytic coordinate for $\tau$.
\end{lemma}
	
	\begin{proof}
If $a^2>4$, the entries of $e^{\tau A}$ are linear combinations of $e^{\lambda_+\tau}$ and $e^{\lambda_-\tau}$, where $\lambda_\pm=(-a\pm\sqrt{a^2-4})/2$. These two exponentials form a global Pfaffian chain of length two and degree one, and the matching equations \eqref{eq:upper-arc-flow} have degree at most two.

If $a^2=4$, then $e^{\tau A}=e^{-a\tau/2}[I+\tau(A+(a/2)I)]$. The single exponential $e^{-a\tau/2}$ forms a global Pfaffian chain of degree one, and the matching equations \eqref{eq:upper-arc-flow} have degree at most three.
		
If $a^2<4$, then
\[
e^{\tau A}=e^{-a\tau/2}
\left(\cos(\omega\tau)I+\frac{\sin(\omega\tau)}{\omega}B\right),
\]
where $\omega=\sqrt{1-a^2/4}$ and $B=A+(a/2)I$.

Let $x(t)$ denote the horizontal coordinate along the orbit segment, parametrized by $0\le t\le\tau$.
Since $x''+ax'+x=0$, the function $x(t)$ is a nonzero sinusoid multiplied by $e^{-at/2}$, and its consecutive zeros are $\pi/\omega$ apart. By assumption, the interior of the orbit segment lies in one open half-plane, so \(x(t)\) has no interior zero. Hence $0<\omega\tau\le\pi$. Introduce $u=\tan((\omega\tau-\pi/2)/2)$ and
\[
h(u)=\frac{1}{1+u^2},\qquad
E(u)=e^{-(a/\omega)\arctan u}.
\]
For the actual orbit segments, \(-1<u\le1\). The functions $(h,E)$ form a global Pfaffian chain on $\mathbb R$ of degree three, since $h'=-2uh^2$ and $E'=-(a/\omega)hE$. After the time reparametrization $u=\tan((\omega\tau-\pi/2)/2)$ and multiplication by $1+u^2$, the matching conditions \eqref{eq:upper-arc-flow} become
\[
(1+u^2)z_1
=e^{-a\pi/(4\omega)}E(u)
\left(-2uI+\frac{1-u^2}{\omega}B\right)z_0.
\]
These two scalar matching equations have total degree at most four, with possible degree-four terms of the form $u^2E(u)Y_0$. Since $1+u^2>0$ on $\mathbb R$, clearing the denominator preserves the real zero set.

Finally, $\tau(u)=(\pi/2+2\arctan u)/\omega$ is analytic with ${\rm d}\tau/{\rm d}u=2/[\omega(1+u^2)]>0$, so $u$ is a local analytic coordinate for $\tau$. In the real-eigenvalue cases, the chosen time parameter is $\tau$ itself. In all three cases, the chain functions depend only on the chosen time parameter, so each chain extends globally to $\mathbb R^3$ by independence of $Y_0,Y_1$.
	\end{proof}

Each orbit segment in the canonical decomposition of a transverse periodic orbit lies in one affine region, with its interior contained in one open half-plane, so Lemma~\ref{lm:upper-arc} applies. 

Fix system~\eqref{sys} and an itinerary from Lemma~\ref{lm:upper-itineraries}, consisting of $k$ orbit segments. Starting at the upper intersection with the $y$-axis and following the flow, denote the successive vertices by 
\[P_j=(\xi_j,Y_j), \qquad j=1,\ldots,k,\]
which satisfy the cyclic condition \(P_{k+1}=P_1\).
The abscissae $\xi_j$ and the affine coefficients along the itinerary are fixed, leaving only the ordinates $Y_j$ and the time parameters as unknowns. 

Applying the construction in the proof of Lemma~\ref{lm:upper-arc} to each orbit segment gives the endpoint matching system
\begin{equation}\label{match}
 \mathcal E_j(Y_j,Y_{j+1},s_j)=0,
\qquad j=1,\ldots,k,\qquad Y_{k+1}=Y_1,   
\end{equation}
where the two components of \(\mathcal E_j\) are
the left-hand sides of the matching equations
constructed in the proof of
Lemma~\ref{lm:upper-arc}, and \(s_j\) is the chosen
time parameter. Adjacent orbit segments
share the same actual vertex ordinate, despite their different
translations.

Since the chain functions for the $j$th orbit segment depend only on $s_j$, concatenating
these chains gives a global Pfaffian representation in
$(Y_1,\ldots,Y_k,s_1,\ldots,s_k)\in\mathbb R^{2k}$.
By Lemmas~\ref{lm:upper-itineraries} and~\ref{lm:upper-arc}, the resulting Pfaffian system may be taken with
\begin{equation}\label{eq:upper-matching-format}
d=2k\le4(n+1),\qquad
\ell\le2k\le4(n+1),\qquad
\alpha=3,\qquad\beta=4.
\end{equation}
Positive passage times and the prescribed itinerary are not
imposed as additional constraints. The following lemma
justifies counting isolated real zeros of  system \eqref{match}.

\begin{lemma}\label{lm:upper-isolation}
For a fixed itinerary associated with the canonical decomposition, the corresponding transverse limit cycles of system~\eqref{sys} determine distinct isolated real zeros of  system \eqref{match}.
\end{lemma}

\begin{proof}
Each transverse limit cycle determines a zero of system
\eqref{match}. The prescribed starting point and cyclic
ordering, together with the flow parametrizations in
Lemma~\ref{lm:upper-arc}, make this representation unique.
Hence distinct cycles determine
distinct zeros by uniqueness of solutions.

Let \(\mathbf z^*\in\mathbb R^{2k}\) be the zero of system \eqref{match}
corresponding to a transverse limit cycle \(\gamma\), with
vertices \((\xi_j,Y_j^*)\) and passage times \(\tau_j^*>0\).
Let $X_j(t;Y)$ denote the horizontal coordinate of the
$j$-th affine flow starting at $(\xi_j,Y)$. Since \(\gamma\) is transverse and \(F\) is continuous,
\[
\partial_t X_j(\tau_j^*;Y_j^*)
=Y_{j+1}^*-F(\xi_{j+1})\neq0.
\]
By the implicit function theorem, the equation $X_j(\tau;Y)=\xi_{j+1}$ uniquely determines a passage time $\tau$ near $\tau_j^*$ for each $Y$ near $Y_j^*$. By Lemma~\ref{lm:upper-arc}, this also determines a unique nearby value of $s_j$. After shrinking the neighborhood, the passage time remains positive.

After removing short time intervals near the endpoints, the remaining portion of each reference orbit segment is compact and has positive distance from all cutting lines. Continuous dependence preserves this separation, while the transversality of $\gamma$ preserves the prescribed crossing directions near the endpoints. Hence every sufficiently nearby zero of \eqref{match} gives orbit segments with the same itinerary. 

Their endpoints agree by \eqref{match}, so continuity of $F$ makes their concatenation a periodic solution of system~\eqref{sys}, which is nonconstant because its initial crossing of the $y$-axis remains transverse. These periodic orbits converge uniformly to $\gamma$ as the corresponding zeros approach $\mathbf z^*$, after parametrizing each orbit segment on $[0,1]$.

If $\mathbf z^*$ were not isolated in $\mathbb R^{2k}$, there would
exist zeros $\mathbf z_m\neq \mathbf z^*$ with $\mathbf z_m\to \mathbf z^*$. The corresponding periodic orbits are distinct from \(\gamma\),
since the zero of \eqref{match} corresponding to \(\gamma\) is unique.
They also converge to \(\gamma\), which contradicts the fact that
\(\gamma\) is a limit cycle.
Thus $\mathbf z^*$ is an isolated real zero.

\end{proof}	
\begin{proof}[\bf{Proof of Theorem~\ref{thm:upper}}]
Fix $n\ge1$ and an admissible function $F$ with exactly $n$ fold points. Let $N(F)$ denote the number of limit cycles of
system~\eqref{sys} with this choice of $F$,
initially allowing $N(F)=+\infty$.

For each itinerary, Lemma~\ref{lm:upper-isolation} assigns distinct transverse limit cycles to distinct isolated real zeros of the matching system. 
Applying the bound \eqref{upper-isolated-bound} in
Lemma~\ref{lm:upper-isolated} with the parameters in
\eqref{eq:upper-matching-format}, we have $D=10$,
and the number of isolated real zeros is at most
\[
		2^{\ell(\ell-1)/2}10^d
		\bigl(3\min\{d,\ell\}+9d+1\bigr)^\ell\leq Q(n) =2^{8(n+1)^2-2(n+1)}
10^{4(n+1)}(48n+49)^{4(n+1)}.
		\]
Summing over the itineraries and adding the tangent limit cycles, Lemma \ref{lm:upper-itineraries} yields
\begin{equation}\label{eq:upper-total}
N(F)\le n+
\left\lfloor\frac{(n+2)^2}{4}\right\rfloor Q(n).
\end{equation}
Taking the supremum over all admissible \(F\) in
\eqref{eq:upper-total}, we obtain
\[
H_n\le
n+\left\lfloor\frac{(n+2)^2}{4}\right\rfloor Q(n).
\]
Using $10(48n+49)<512(n+1)$ and $\log_2(n+1)\le n$,
the preceding bound yields
\[
H_n\le2(n+1)^2Q(n)
<2^{12n^2+56n+43}
<2^{28(n+1)^2}.
\]
\end{proof}

\appendix
\appendixsectionformat
\renewcommand{\thesection}{\Alph{section}}
\section{}
\label{app:node-polynomials}

This appendix constructs the node polynomials used in
Appendices~\ref{app:fourier-calculations}
and~\ref{app:constrained-derivatives}.
Throughout, the parameters are fixed at \eqref{trig-base}.
Recall that
\[
\vartheta=\frac{\pi}{2n},
\qquad
h=\frac{\pi}{4n},
\qquad
s=\sin\frac{\pi}{4n},
\qquad
c=\cos\frac{\pi}{4n},
\qquad
q=\cos\frac{\pi}{2n},
\qquad
A=a_n=c.
\]
Put \(\sigma_i=(-1)^{i-1}\).

For each integer \(r\ge0\), define
\[
J_{r,D}(x)=\int_0^x t^rD(t)\,\dd t,
\qquad
J_{r,S}(x)=\int_0^x t^rS(t)\,\dd t.
\]
For \(0\le r\le2n-6\), there are polynomials
\(R_{r,D}\) and \(R_{r,S}\), of degree at most \(r+3\)
and of the same parity as \(r+3\), such that
\begin{equation}\label{node-antiderivative}
	J_{r,D}(a_i)=R_{r,D}(a_i)+C_{r,D},
	\qquad
	J_{r,S}(a_i)=\sigma_iR_{r,S}(a_i)+C_{r,S}.
\end{equation}

Indeed, with
\(a_\pm=\sin(\tau\pm h)\) and
\(\Delta_\nu=a_+^\nu-a_-^\nu\),
the required polynomials are obtained from
\[
R_{r,D}(a_+)-R_{r,D}(a_-)
=
\frac{\sin^2\tau}{s(r+1)}\Delta_{r+1}
-\frac{2c\sin\tau}{s(r+2)}\Delta_{r+2},
\]
and
\[
R_{r,S}(a_+)+R_{r,S}(a_-)
=
2\cos\tau
\left(
\frac{\Delta_{r+2}}{r+2}
-\frac{c\sin\tau}{q(r+1)}\Delta_{r+1}
\right).
\]
For the nonzero frequencies \(1\le\nu\le r+3\le2n-3\),
the difference and sum multipliers are
\(2\mathrm{i}\sin(\nu h)\) and \(2\cos(\nu h)\),
respectively, and are nonzero.
The difference equation has zero constant Fourier coefficient
on its right-hand side, so the constant term of \(R_{r,D}\),
when present, may be chosen freely.

Taking \(\tau=j\vartheta\) and using \eqref{SDj} gives
\eqref{node-antiderivative}.
Taking \(\tau=\pi/2\) yields \(R_{r,S}(A)=0\).
For odd \(r\), taking \(\tau=0\) further gives
\(C_{r,S}=0\).

Write \(R_r=R_{r,D}+R_{r,S}\) and \(d=r+3\).
Comparison of the highest Fourier coefficients gives
\[
[x^d]R_{r,D}
=
\frac{
	\sin((d-2)h)/(d-2)
	-2c\sin((d-1)h)/(d-1)
}{
	s\sin(dh)
},
\]
and
\[
[x^d]R_{r,S}
=
\frac{
	2\bigl(
	(c/q)\sin((d-2)h)/(d-2)
	-\sin((d-1)h)/(d-1)
	\bigr)
}{
	\cos(dh)
}.
\]
Adding the two expressions and simplifying yields
\begin{equation}\label{app:R-leading}
	[x^{r+3}]R_r=L(r+3),
	\qquad
	L(x):=\frac{F_*(x)}{x},
\end{equation}
where \(F_*\) is defined in the proof of
Lemma~\ref{lm:rank}.

\section{}
\label{app:fourier-calculations}

This appendix supplies the integral and Fourier calculations
used in the proof of Lemma~\ref{lm:degeneracy}.
Throughout, the parameters are fixed at \eqref{trig-base}.
We retain the notation of that proof and recall that
\[
\vartheta=\frac{\pi}{2n},
\qquad
h=\frac{\pi}{4n},
\qquad
s=\sin\frac{\pi}{4n},
\qquad
q=\cos\frac{\pi}{2n}.
\]
We first establish the Fourier degree bound for
\(N_{m\ell,j}\), then compute the local integral involving
\(S D\), and finally derive the critical-order expression
for \(N_{m\ell}\).

We begin with the interval representations of
\(I_{\ell,S}\) and \(I_{\ell,D}\).
Put \(r=2\ell+1\) and \(d=2\ell+4\).
We use \(v\) as a continuous Fourier variable, evaluated
at \(v=j\vartheta\) on \(I_j\). Define
\[
\begin{aligned}
	U_S(x,v)
	&=
	2\cos v
	\left(
	\frac{x^{d-1}}{d-1}
	-\frac{\cos h\sin v}{q(d-2)}x^{d-2}
	\right),\\
	U_D(x,v)
	&=
	\frac1s
	\left(
	\frac{\sin^2v}{d-2}x^{d-2}
	-\frac{2\cos h\sin v}{d-1}x^{d-1}
	\right).
\end{aligned}
\]
On \(I_j\), integration of \eqref{SDj} gives
\[
I_{\ell,S}(x)=(-1)^j\bigl(U_S(x,v)+\mathcal C_\ell(v)\bigr),
\qquad
I_{\ell,D}(x)=U_D(x,v)+\mathcal D_\ell(v).
\]
For \(j\ge1\), evaluation at \(a_j=\sin(v-h)\), using
\eqref{node-antiderivative} from
Appendix~\ref{app:node-polynomials}, yields
\[
\begin{aligned}
	\mathcal C_\ell(v)
	&=
	-R_{r,S}(\sin(v-h))
	-U_S(\sin(v-h),v),\\
	\mathcal D_\ell(v)
	&=
	R_{r,D}(\sin(v-h))+C_{r,D}
	-U_D(\sin(v-h),v).
\end{aligned}
\]
The same representations hold on the remaining intervals
by symmetry. Since \(R_{r,S}\) and \(R_{r,D}\) are even
polynomials of degree at most \(d\), both
\(\mathcal C_\ell\) and \(\mathcal D_\ell\) are
\(\pi\)-periodic trigonometric polynomials of Fourier
degree at most \(d\).
Consequently, after extracting \((-1)^j\) from
\(I_{\ell,S}\), every term \(x^pT(v)\) satisfies
\[
p+\deg_{\rm F}T\le d=2\ell+4.
\]
Here the Fourier degree is understood after factoring out
\((-1)^j\). Since \(m+\ell=k-1\), it follows that
\begin{equation}\label{app:N-fourier-degree}
\deg_{\rm F}N_{m\ell,j}
\le2(2\ell+4)+2(m-\ell)
=2k+6.
\end{equation}

For the local term, integration of \eqref{SDj} over the
successive intervals gives, for every positive odd integer \(t\),
\[
\int_0^{a_n}x^t S(x)D(x)\,\dd x
=
\frac{s(1+2q)}{q(t+1)(t+2)}M_{t+2}
-
\left(
\frac{1}{sq(t+1)(t+2)}
-
\frac{4q}{s(t+1)(t+2)(t+3)}
\right)M_{t+4}.
\]
Taking \(t=2n-5\) and using \(M_{2n-3}=0\), we obtain
\begin{equation}\label{A-critical-Q}
-Q_{n-3}
=
\frac{n-1-2q^2}
{sq(n-1)(2n-4)(2n-3)}M_{2n-1}.
\end{equation}

We now assume \(m>\ell\) and \(m+\ell=n-4\).
Let \(f_S,f_D,H_S,H_D\) be as defined in the proof of
Lemma~\ref{lm:degeneracy}.
To identify the highest Fourier coefficients of
\(I_{\ell,S}\) and \(I_{\ell,D}\), set
\[
\mathcal A_S(v,t)
=
U_S(\sin(v+t),v)+\mathcal C_\ell(v),
\qquad
\mathcal A_D(v,t)
=
U_D(\sin(v+t),v)+\mathcal D_\ell(v).
\]
Substitution of the leading coefficients of
\(R_{r,S}\) and \(R_{r,D}\), computed in
Appendix~\ref{app:node-polynomials}, gives
\begin{equation}\label{antiderivative-top-frequency}
	[e^{idv}]\mathcal A_S(v,t)
	=
	\frac{f_S(e^{it})}{(2i)^{d-1}},
	\qquad
	[e^{idv}]\mathcal A_D(v,t)
	=
	-\frac{f_D(e^{it})}{s(2i)^d},
\end{equation}
where \([e^{idv}]\) denotes the coefficient of
\(e^{idv}\) in the Fourier expansion with respect to \(v\).

Since \(2d+2m-2\ell=2n\), we have
\[
[e^{i(2m-2\ell)v}]
\bigl(\sin^{2m-2\ell-1}(v+t)\cos(v+t)\bigr)
=
\frac{e^{i(2m-2\ell)t}}{2(2i)^{2m-2\ell-1}}.
\]
Thus, under \(x=\sin(v+t)\), the coefficient of
\(e^{2niv}\) in the transformed integrand is, by
\eqref{antiderivative-top-frequency},
\[
\begin{aligned}
	\frac{(-1)^jf_S(e^{it})}{(2i)^{d-1}}
	\left(-\frac{f_D(e^{it})}{s(2i)^d}\right)
	\frac{e^{i(2m-2\ell)t}}{2(2i)^{2m-2\ell-1}}=
	\frac{(-1)^{j+n}}{s\,2^{2n-1}}
	e^{2nit}H_S(t)H_D(t).
\end{aligned}
\]
Adding the conjugate negative-frequency term and
integrating over \([-h,h]\), the contribution of the
frequencies \(\pm2n\) to the integral over \(I_j\) is
\[
\frac{(-1)^{j+n}}{s\,2^{2n-2}}
\operatorname{Re}\left(
e^{2nij\vartheta}
\int_{-h}^{h}
e^{2nit}H_S(t)H_D(t)\,\dd t
\right).
\]
All lower-frequency contributions vanish upon summation
by \eqref{fourier-cancellation}. Since
\(e^{2nij\vartheta}=e^{i\pi j}=(-1)^j\)
and
\[
\operatorname{Re}\int_{-h}^{h}
e^{2nit}H_S(t)H_D(t)\,\dd t=2\mathcal I,
\]
summing over the \(2n\) intervals and using
\(2N_{m\ell}=\sum_jN_{m\ell,j}\), we obtain
\begin{equation}\label{app:critical-N}
N_{m\ell}
=
(-1)^n\frac{2n}{s\,2^{2n-2}}\mathcal I.
\end{equation}

\section{}
\label{app:constrained-derivatives}

This appendix supplies the calculations used in the proof of
Lemma~\ref{lm:rank}. All parameter derivatives are evaluated at the
parameter values given in \eqref{trig-base}.
We first compute partial derivatives with \(\mathbf e\)
and \(\mathbf a\) independent, and then use
\eqref{moment-condition} to express the variations in
\(\mathbf e\) in terms of those in \(\mathbf a\).
This gives \(D_i\), the derivative with respect to
\(a_i\) along \(\mathcal M\).
Recall that
\[
h=\frac{\vartheta}{2},\qquad
s=\sin h,\qquad
c=\cos h,\qquad
q=\cos\vartheta,\qquad
A=a_n=c,
\]
and put \(\sigma_i=(-1)^{i-1}\).

We also use the differentiated defining equations of
\(\mathcal M\). For every fixed
\(\Lambda\in\operatorname{span}\{1,x,x^3,\ldots,x^{2n-3}\}\),
they give
\begin{equation}\label{app:moment-cancellation}
	\sum_i\Lambda(a_i)\,\delta e_i
	=
	-\sum_i e_i\Lambda'(a_i)\,\delta a_i.
\end{equation}

We first consider \(Q_k\), \(0\le k\le n-4\). Recall that
\[
Q_k=\int_0^A x^{2k+1}S(x)D(x)\,\dd x.
\]
Differentiation with independent parameters gives
\[
\partial_{e_i}Q_k
=
\int_{a_i}^A
x^{2k+1}(x-a_i)(D+\sigma_iS)\,\dd x,
\qquad
\frac{\partial_{a_i}Q_k}{e_i}
=
-\int_{a_i}^A
x^{2k+1}(D+\sigma_iS)\,\dd x.
\]
The moving-endpoint term vanishes since
\(S(A)=(2+M_0)A-M_1=0\).
Substituting \eqref{node-antiderivative} and using
\eqref{app:moment-cancellation} gives
\[
	\delta Q_k
	=
	\sum_i e_i
	\left(
	R_{2k+2}'(a_i)-a_iR_{2k+1}'(a_i)
	\right)\delta a_i.
\]
The polynomial in parentheses is even and has degree at most
\(2k+4\). Hence
\begin{equation}\label{app:Q-gradient}
D_iQ_k=e_iP_Q(a_i^2),
\qquad
\deg P_Q\le k+2.
\end{equation}
By \eqref{app:R-leading}, the coefficient of
\(z^{k+2}\) in \(P_Q(z)\) is
\begin{equation}\label{app:kappa-Q}
	\kappa_Q
	:=[z^{k+2}]P_Q(z)
	=
	F_*(2k+5)-F_*(2k+4).
\end{equation}

We next consider \(N_{m\ell}\), where
\(m+\ell=k-1\) and \(m>\ell\). Put \(r=2\ell+1\) and \(p=2(m-\ell)-1\), so that
\[
N_{m\ell}
=
\int_0^A x^pJ_{r,S}(x)J_{r,D}(x)\,\dd x.
\]
Define
\[
K_E(t)=\int_t^A x^pJ_{r,E}(x)\,\dd x,
\qquad
H_{z,E}(a)=\int_0^a t^zK_E(t)\,\dd t,
\qquad E\in\{D,S\}.
\]
Differentiation and interchange of the order of integration give
\[
\partial_{e_i}N_{m\ell}
=
\int_{a_i}^A
t^r(t-a_i)(K_D+\sigma_iK_S)\,\dd t,
\qquad
\frac{\partial_{a_i}N_{m\ell}}{e_i}
=
-\int_{a_i}^A
t^r(K_D+\sigma_iK_S)\,\dd t.
\]
The moving-endpoint term vanishes because
\(J_{r,S}(A)=0\). Moreover, the factor \(2\) from differentiating
\(I_{\ell,R}^2-I_{\ell,L}^2\) is already contained here, since
\(J_{r,D}+\sigma_iJ_{r,S}\) equals \(2I_{\ell,R}\) or
\(-2I_{\ell,L}\), according to the side.

Two integrations by parts give
\begin{equation}\label{app:H-identity}
	\begin{aligned}
		H_{z,E}(a)
		=
		\frac{K_E(0)a^{z+1}}{z+1}
		-\frac{a^{z+p+2}J_{r,E}(a)}
		{(p+1)(z+p+2)}
		+
		\frac{a^{z+1}J_{r+p+1,E}(a)}
		{(p+1)(z+1)}
		-\frac{J_{r+z+p+2,E}(a)}
		{(z+p+2)(z+1)}.
	\end{aligned}
\end{equation}
Indeed,
\[
K_E(t)
=
K_E(0)
-\frac{t^{p+1}J_{r,E}(t)}{p+1}
+\frac{J_{r+p+1,E}(t)}{p+1}.
\]

Substituting \eqref{node-antiderivative} into
\eqref{app:H-identity}, for \(z=r,r+1\), gives
\[
H_{z,D}(a_i)
=\widehat R_{z,D}(a_i)+\widehat C_{z,D},
\qquad
H_{z,S}(a_i)
=\sigma_i\widehat R_{z,S}(a_i)+\widehat C_{z,S}.
\]
Since \(r\) and \(r+p+1\) are odd, the node identities give
\[
C_{r,S}=C_{r+p+1,S}=0,
\qquad
J_{r,S}(A)=J_{r+p+1,S}(A)=0,
\]
and hence \(K_S(0)=0\). Thus \eqref{app:H-identity} yields
\[
\widehat R_{r,S}(A)=\widehat R_{r+1,S}(A)=0,
\qquad
\widehat C_{r,S}
=-\frac{C_{2r+p+2,S}}{(r+p+2)(r+1)}=0.
\]
Writing \(\widehat R_z=\widehat R_{z,D}+\widehat R_{z,S}\),
we find that \(\widehat R_r\) and \(\widehat R_{r+1}\)
are respectively even and odd, with degrees at most
\(2k+4\) and \(2k+5\).

Set \(\widehat b_j=\widehat R_{r+j,D}(A)\), \(j=0,1\).
The partial derivative formulas become
\[
\partial_{e_i}N_{m\ell}=\widehat\Lambda(a_i),
\qquad
\frac{\partial_{a_i}N_{m\ell}}{e_i}
=-\widehat b_0+\widehat R_r(a_i),
\]
where
\(\widehat\Lambda(x)
=\widehat b_1-x\widehat b_0
-\widehat R_{r+1}(x)+x\widehat R_r(x)\).
This is a constant plus an odd polynomial of degree
at most \(2k+5\le2n-3\). Hence
\eqref{app:moment-cancellation} gives
\[
\delta N_{m\ell}
=\sum_i e_i
\bigl(\widehat R_{r+1}'(a_i)-a_i\widehat R_r'(a_i)\bigr)
\delta a_i,
\]
and therefore
\begin{equation}\label{app:N-gradient}
D_iN_{m\ell}=e_iP_N(a_i^2),
\qquad
\deg P_N\le k+2.
\end{equation}

It remains to compute the coefficient of degree \(k+2\)
in \(P_N\). By \eqref{app:H-identity} and
\eqref{app:R-leading},
\[
\begin{aligned}
[x^{r+z+p+5}]\widehat R_z
&=
-\frac{L(r+3)}{(p+1)(z+p+2)}
+\frac{L(r+p+4)}{(p+1)(z+1)}
-\frac{L(r+z+p+5)}{(z+p+2)(z+1)}\\
&=-L[2\ell+4,\,2m+4,\,r+z+p+5].
\end{aligned}
\]
Taking \(z=r,r+1\) and using \(xL[a,b,x]=F_*[a,b,x]-L[a,b]\),
we obtain
\begin{equation}\label{app:kappa-N}
\kappa_N
:=[z^{k+2}]P_N(z)
=-\Bigl(
F_*[2\ell+4,\,2m+4,\,2k+5]
-F_*[2\ell+4,\,2m+4,\,2k+4]
\Bigr).
\end{equation}

	\section*{Acknowledgements}
All authors are supported by the  National Natural Science Foundation of China (Nos. 12571190, 12322109) and  Hunan Basic Science Research Center for Mathematical Analysis (2024JC2002).


\begin{thebibliography}{79}	
		
		\bibitem{Arnold}
		V. I. Arnold,
		{\it Geometric methods in the theory of ordinary differential equations, 2nd ed.}, Springer-Verlag, New York, 1987.
		
		\bibitem{BuicaLlibre}
		A. Buic\u{a}, J. Llibre,
		Averaging methods for finding periodic orbits via Brouwer degree,
		{\it Bull. Sci. Math.}, {\bf 128} (2004), 7--22.

        \bibitem{BNY}
G. Binyamini, D. Novikov, S. Yakovenko, On the number of zeros of Abelian integrals: A constructive solution of the infinitesimal Hilbert sixteenth problem,
{\it Invent. Math.}, {\bf 181} (2010), 227--289.
		
		\bibitem{CFJM}
		H. Chen, Z. Feng, M. Jia, Y. Meng, Maximum number of limit cycles of the piecewise linear Li\'enard system with three zones, \textit{J. Lond. Math. Soc. (2)}, \textbf{113} (2026), Paper No. e70430.
		
		\bibitem{CJT}
		H. Chen, M. Jia, Y. Tang, Topological classifications of a piecewise linear Li\'enard system with three zones,  {\it J. Differential Equations},  {\bf 399} (2024), 1--47.
		
		\bibitem{CLXY}
		H. Chen, D. Li, J. Xie, Y. Yue, Limit cycles in planar continuous piecewise linear systems, \textit{Commun. Nonlinear Sci. Numer. Simul.}, \textbf{47} (2017), 438--454.
		
		\bibitem{CLZZ}
		H. Chen, Z. Li, R. Zhang,  X. Zhang,
		The number of limit cycles of piecewise linear Li\'enard systems,
		arXiv:2608.19542, 2026.
		
		
      
        
		\bibitem{DM}
		P. De Maesschalck,  F. Dumortier,
		Classical Li\'enard equations of degree \(n\geq6\) can have \([(n-1)/2]+2\) limit cycles,
		{\it J. Differential Equations}, {\bf 250} (2011), 2162--2176.
          \bibitem{DM2}
 P. De Maesschalck, R. Huzak, Slow divergence integrals in classical Li\'enard equations near centers,  {\it J.
Dynam. Differential Equations}, {\bf 27} (2015), no. 1, 177–185.
		\bibitem{DPR}
		F. Dumortier, D. Panazzolo, R. Roussarie,
		More limit cycles than expected in Li\'enard equations,
		{\it Proc. Amer. Math. Soc.}, {\bf 135} (2007), 1895--1904.

       
		
        \bibitem{Gab}
		A. Gabrielov, N. Vorobjov, Complexity of computations with Pfaffian and Noetherian functions, in {\it Normal forms, bifurcations and finiteness problems in differential equations},  NATO Sci. Ser. II Math. Phys. Chem., 137, Kluwer Acad. Publ., Dordrecht, 2004, 211--250.
		
		\bibitem{H}
		D. Hilbert, Mathematical problems, {\it Bull. Amer. Math. Soc.}, {\bf 8} (1902), 437–479.

        \bibitem{K}
A. G. Khovanskii, Real analytic manifolds with the property of finiteness, and complex Abelian integrals,
{\it Funct. Anal. Appl.}, {\bf 18} (1984), 40--50.


		\bibitem{LL}
		C. Li,  J. Llibre,
		Uniqueness of limit cycles for Li\'enard differential equations
		of degree four,
		{\it J. Differential Equations}, {\bf 252} (2012), 3142--3162.
\bibitem{LL2}
        C. Li, K. Lu, Slow divergence integral and its application to classical Li\'enard equations of degree 5, {\it J. Differential Equations},  {\bf 257} (2014), 4437–4469.
		
		\bibitem{LMP}
		A. Lins, W. de Melo, C. C. Pugh, On Li\'enard's equation, {\it Lecture Notes in Math.}, {\bf 597} (1977), 335–357.
		
		
		
		\bibitem{LPV}
		J. Llibre, E. Ponce, C. Valls, Uniqueness and nonuniqueness of limit cycles for piecewise linear differential systems with three zones and no symmetry, {\it J. Nonlinear Sci.}, {\bf 25} (2015), 861--887.
		
		\bibitem{LPV2}
		J. Llibre, E. Ponce,  C. Valls,
		Two limit cycles in Li\'enard piecewise linear differential
		systems,
		{\it J. Nonlinear Sci.}, {\bf 29} (2019), 1499--1522.
		
		
		\bibitem{LZ}
		J. Llibre, X. Zhang,
		Limit cycles of the classical Li\'enard differential systems:
		a survey on the Lins Neto, de Melo and Pugh's conjecture,
		{\it Expo. Math.}, {\bf 35} (2017), 286--299.
		
		\bibitem{M1}
		H. P. McKean,  Nagumo's equation,
		{\it Adv. Math.}, {\bf 4} (1970), 209--223.
		
		\bibitem{M2}
		H. P. McKean,
		Stabilization of solutions of a caricature of the
		FitzHugh--Nagumo equation,
		{\it Comm. Pure Appl. Math.}, {\bf 36} (1983), 291--324.        
		
		
		\bibitem{Smale}
		S. Smale,
		Dynamics retrospective: great problems, attempts that failed,  {\it Phys. D}, {\bf 51} (1991), 267--273.
		
		\bibitem{S2}
		S. Smale, Mathematical problems for the next century, {\it Math. Intell.},  {\bf 20} (1998), 7–15.
		
		
		\bibitem{T}
		A. Tonnelier, The McKean’s caricature of the FitzHugh-Nagumo model. I. The space–clamped system,
		{\it SIAM J. Appl. Math.},  {\bf 63} (2002), 459–484.
		
		\bibitem{V}
A. N. Varchenko, Estimation of the number of zeros of an Abelian integral depending on a parameter, and limit cycles,
{\it Funct. Anal. Appl.}, {\bf 18} (1984), 14--25.
		
		
		
		
		
		
		
		
		
		
	\end{thebibliography}
\end{document}